\documentclass[11pt,a4]{amsart}
 \usepackage{epsfig}
\usepackage{amssymb}
\usepackage{color}
\usepackage{mathrsfs}
\usepackage{hyperref}

\usepackage{bbm}

\def\1{\mathbbm{1}}

\newcommand\R{{\mathbb R}}

\def\AA{{\mathcal A}}
\def\BB{{\mathcal B}}

\def\FF{{\mathcal F}}
\def\GG{{\mathcal G}}

\def\LL{{\mathcal L}}
\def\MM{{\mathcal M}}

\def\TT{{\mathcal T}}

\DeclareMathOperator{\sign}{sign}

\newcommand{\sk}{\smallskip}

\let\oldmarginpar\marginpar
\renewcommand\marginpar[1]{\-\oldmarginpar[\raggedleft\footnotesize #1]%
{\raggedright\footnotesize #1}}

\newtheorem{theo}{Theorem}

\newtheorem{lem}[theo]{Lemma}

\theoremstyle{definition}

\theoremstyle{remark}
\newtheorem{rem}[theo]{Remark}

\newcommand{\beqn}{\begin{equation}}
\newcommand{\eeqn}{\end{equation}}
\newcommand{\bear}{\begin{eqnarray}}
\newcommand{\eear}{\end{eqnarray}}
\newcommand{\bean}{\begin{eqnarray*}}
\newcommand{\eean}{\end{eqnarray*}}

\def\signcc{\bigskip \begin{center} {\sc Chuqi Cao
      \par\vspace{3mm} 
      BISMA, Tsinghua University, and \par
      CEREMADE, Universit\'e Paris-Dauphine, \par
      Place du Mar\'echal de Lattre de Tassigny \par
      75775 Paris Cedex 16 FRANCE   \par
      OCDID: 0000-0002-5451-9962\par
      \vspace{3mm} e-mail:}
    \tt{chuqicao@gmail.com} \end{center}}

\begin{document}

\title[The KFP Equation with general force]{The Kinetic Fokker-Planck Equation with General Force}

\author{Chuqi Cao }

\maketitle

\begin{center} {\bf 
\today}
\end{center}

\begin{abstract} 

We consider the kinetic Fokker-Planck equation with a class of general force. We prove the existence and uniqueness of a positive normalized equilibrium (in the case of a general force) and  establish
some exponential rate of convergence to the equilibrium (and the rate can be explicitly computed).  Our results improve results about classical force to general force case. Our result also improve the rate of convergence for the Fitzhugh-Nagumo equation from  non-quantitative to quantitative explicit rate.

  \end{abstract}

\bigskip

\tableofcontents

\section{Introduction}
\label{sec1}
\setcounter{equation}{0}
\setcounter{theo}{0}

In this paper, we consider the kinetic Fokker-Planck (KFP for short) equation with general force and confinement
\sk
\beqn\label{E11}
\partial_t f =\LL f := -v \cdot \nabla_x f +\nabla_x V(x) \cdot \nabla_v f +\Delta_v f + \hbox{div}_v(\nabla_v W(v) f),
\eeqn
for a density function $f = f(t, x, v)$, with $t \ge 0 ,\ x \in \R^d,\ v \in \R^d$, with
\beqn
\nonumber
V(x)=\frac { \langle x \rangle^{\gamma} } {\gamma},\quad \gamma \ge 1 ,\quad W(v)=\frac{ \langle v \rangle^{\beta} } {\beta}, \quad \beta \ge 2,
\eeqn
where $ \langle x \rangle^2 := 1 + |x|^2 $, and  the Fitzhugh-Nagumo equation 
\beqn\label{E12}
\partial_t f :=\LL f  = \partial_x(A(x, v) f ) + \partial_v(B(x, v ) f ) +  \partial^2_{vv} f,
\eeqn
with 
\beqn
\nonumber
A(x, v) = ax- bv, \quad B(x, v) =v(v-1) (v-c) +x,
\eeqn
for some $a, b, c>0$.
The evolution equations are complemented with an initial datum
\beqn
\nonumber
f(0, x, v) = f_0(x, v ) \quad \hbox{on} \ \R^{2d} .
\eeqn
It's easily seen that both equations are mass conservative,  that is
\beqn
\nonumber
\MM( f(t,\cdot ))= \MM (f_0), 
\eeqn
where we define the mass of $f$ by
\beqn
\nonumber
\MM(f) =\int_{\R^d \times \R^d} f(x, v) dx dv.
\eeqn
When $G$ satisfies 
\beqn
\nonumber
\LL G=0, \quad \MM(G)=1, \quad G > 0,
\eeqn
we say that $G$ is a positive normalized steady state.\\
For a given weight function $m$, we will denote $L^p(m) = \{ f | fm \in L^p \}$ the associated Lebesgue space and $\Vert f \Vert_{L^p(m)} = \Vert f m \Vert_{L^p}$ the associated norm, for $p=2$ we also use $(f, g)_{L^2(m)}$ to denote the associate scalar product.

\noindent With these notations, we can introduce the main result of this paper.

\begin{theo}\label{T11}

(1)  When $2 \le \beta, 1 \le \gamma$, there exist a weight function $m>0$ and a positive normalized steady state $G\in L^1(m)$  such that for any initial datum $f_0 \in L^p(m), p \in [1, \infty]$, the associated solution $f(t, \cdot)$ of the kinetic Fokker-Planck equation (\ref{E11}) satisfies
\beqn
\nonumber
\Vert f(t , \cdot)	- \MM(f_0) G\Vert_{L^p(m)}  \le C e^{- \lambda t } \Vert f_0 -\MM(f_0) G\Vert_{L^p(m)}  ,
\eeqn
for some constant $C, \lambda >0$.\\
(2)  The same conclusion holds for the kinetic Fitzhugh-Nagumo equation (\ref{E12}).
\end{theo}

In the results above the constants $C$ and $\lambda$ can be explicitly estimated in terms of
the parameters appearing in the equation by following the calculations in the proofs.
We do not give them explicitly since we do not expect them to be optimal, but they
are nevertheless completely constructive.

\begin{rem}

Theorem \ref{T11} is also true when $V(x)$ behaves like $ \langle x \rangle^{\gamma} $ and $W(v)$ behaves like $ \langle v \rangle^{\beta} $, that is for any $V(x)$ satisfying

\beqn
\nonumber
C_1 \langle x \rangle^{\gamma} \le V(x) \le C_2\langle x \rangle^{\gamma} , \ \ \forall x \in \R^d ,
\eeqn

\beqn
\nonumber
C_3|x| \langle x \rangle^{\gamma-1} \le x \cdot \nabla_x V(x) \le C_4|x|\langle x \rangle^{\gamma-1}, \ \ \forall x \in B_R^c ,
\eeqn
and 
\beqn
\nonumber
|D^n_x V(x)| \le C_n \langle x \rangle^{\gamma-2},\quad \forall x \in \R^d , \quad  \forall n \ge 2,
\eeqn
for some constant $C_i >0$, $R>0$, and similar estimates holds for $W(v)$.
\end{rem}

We prove both cases of Theorem \ref{T11} by proving the following theorem, which gives convergence result for more general KFP type models.

\begin{theo}\label{T12}
Consider the following equation
\beqn\label{E13}
\partial_t f :=\LL f  = \textup{div}_x(A(x, v) f ) + \textup{div}_v(B(x, v ) f ) +  K\Delta_v f,
\eeqn
with $K>0$ constant, $A(x, v), B(x, v) \in C^1$ and 
\beqn
\nonumber
A(x, v) = -v+\Phi(x),
\eeqn
 where $\Phi(x)$ is Lipschitz
\beqn
\nonumber
|\Phi(x) - \Phi(y)| \le M |x-y|,
\eeqn
for some $M>0$. We assume also that there exist $W(x, v)$ such that
\beqn
\nonumber
\nabla_vW(x, v) = B(x, v),
\eeqn
define
\bear
\nonumber
\phi_2(m)  &=& v \cdot \frac {\nabla_x m} m - \Phi(x) \cdot \frac {\nabla_x m} {m} + \frac 1 2 \textup{div}_x \Phi(x) + K\frac {|\nabla_v m|^2} {m^2}
\\ \nonumber
 &+&K\frac{\Delta_v m} {m} - B(x, v) \cdot \frac{\nabla_v m } {m} + \frac 1 2 \textup{div}_v B(x, v).
\eear
If we can find a weight function $m$ and a function $H \ge 1$ such that the four conditions holds\\
(C1)(Lyapunov condition) For some $\alpha, b>0$ there holds
\beqn
\nonumber
\LL^*m \le -\alpha m +b,
\eeqn
(C2)for some constant $C_1, C_2, C_3 >0$ we have
\beqn
\nonumber
-C_1 H \le \phi_2(m)  \le -C_2 H +C_3,
\eeqn
(C3)For any integer $n \ge 1$ fixed, for any $\epsilon>0$ small, we can find a constant $C_{\epsilon, n} $ such that
\beqn
\nonumber
\sum_{k=1}^n |D_x^k \Phi(x)| +\sum_{k=1}^n |D_{x, v}^k B(x, v)| \le C_{n,\epsilon} + \epsilon H,
\eeqn
(C4)For some constant $C_4>0$ there holds
\beqn
\nonumber
\frac {\Delta_{x, v} m} {m } \ge -C_4.
\eeqn
Then there exist a positive normalized steady state $G$ such that
\beqn
\nonumber
\Vert f(t , \cdot)	- \MM(f_0) G\Vert_{L^1(m)}  \le C e^{- \lambda t } \Vert f_0 -\MM(f_0) G\Vert_{L^1(m)},
\eeqn
for some $C, \lambda>0$. In addition,  for any $p \in[1, \infty]$, if 
\beqn
\nonumber
\varphi_p(m) \le -a +M \1_{B_R},
\eeqn
for some constant $a, M, R>0$, where
\bear
\nonumber
\varphi_p(m)  &=& v \cdot \frac {\nabla_x m} m - \Phi(x) \cdot \frac {\nabla_x m} {m} + ( 1 -\frac 1 p) \textup{div}_x \Phi(x) +  2K(1-\frac 1 p) \frac {|\nabla_v m|^2} {m^2}
\\ \nonumber
 &+& K (\frac 2 p -1)\frac{\Delta_v m} {m} - B(x, v) \cdot \frac{\nabla_v m } {m} + (1-\frac 1 p) \textup{div}_v B(x, v),
\eear
then we have
\beqn
\nonumber
\Vert f(t , \cdot)	- \MM(f_0) G\Vert_{L^p(m)}  \le C e^{- \lambda t } \Vert f_0 -\MM(f_0) G\Vert_{L^p(m)}.
\eeqn
\end{theo}

\begin{rem}\label{R11}
In fact $\phi_2(m)$ satisfies 
\bear
\nonumber
\int_{\R^d \times \R^d }  (f (\LL g) + g (\LL f)) m^2 &=&-2K \int_{\R^d \times \R^d } \nabla_v f \cdot \nabla_v g m^2 + 2\int_{\R^d \times \R^d }  f g \phi_2(m)m^2,
\eear
and $\varphi_p(m)$ satisfies
\bear
\nonumber
 \int_{\R^d \times \R^d } \sign f |f|^{p-1} \LL f m^p=- K\int_{\R^d \times \R^d } |\nabla_v( m f )|^2 |f|^{p-2} m^{p-2} + \int_{\R^d \times \R^d }  |f|^p\varphi_p(m)m^p.
\eear
the computation can be found in Appendix \ref{secAB}. Condition (C2)-(C4) ensures some regularity estimate which we will see in Section \ref{sec3}.
\end{rem}

\begin{rem}\label{R12}
For the kinetic Fokker-Planck equation with general force \ref{E11}, we can take 
\beqn
\nonumber
W(x, v) = F(v) +v \cdot \nabla_x V(x),
\eeqn
and
\beqn
\nonumber
m=e^{\lambda H_1}, \quad  H_1 =|v|^2 + V(x)+ \epsilon v \cdot \nabla_x  \langle x \rangle, \quad H =  \langle v \rangle^\beta + \langle x  \rangle^{\gamma-1}+1,
\eeqn
for some $\lambda, \epsilon>0$ small. For the kinetic Fitzhugh-Nagumo equation (\ref{E12}), we can take
\beqn
\nonumber
m = e^{\lambda(x^2 +v^2)}, \quad H= |v|^4 + |x|^2+1,\quad W = \frac 1 4 |v|^4 - \frac 1 3(1+c) v^3  +  \frac 1 2 |v|^2 + x \cdot v,
\eeqn
for some constant $\lambda>0$, the computation can be found in Section \ref{sec6} below. 
\end{rem}

\begin{rem}

For the kinetic Fitzhugh-Nagumo equation (\ref{E12}), an exponential convergence with non-quantitative rate to the convergence has already been proved in   \cite{MQT}, our method improves the result to a quantitative rate.
\end{rem}

If $\beta = 2$, the equation (\ref{E11}) will turns to the classical KFP equation
\beqn
\nonumber
\partial_t f =\LL f := -v \cdot \nabla_x f +\nabla_x V(x) \cdot \nabla_v f +\Delta_v f + \hbox{div}_v(v f).
\eeqn
This time we observe that 
\beqn
\nonumber
\quad G=Z^{-1}e^{-W}, \quad W= \frac {v^2} {2} +V(x), \quad Z \in \R_{+},
\eeqn
is an explicit steady state. There are many classical results for this equation on the case $ \gamma \ge 1 $, where there is an exponentially decay. We refer the interested readers to \cite{V, DMS, DMS2, HN, HN2, H, BCG,MM}, and for the weak confinement case $\gamma \in(0,1)$, there are also some polynomial or sub-geometric convergence results proved in \cite{BCG,C,DFG}.  We also emphasize that our results for kinetic Fokker-Planck equation with general force  and confinement are to our knowledge new.
\sk

We carry out all of our proofs using variations of Harris's Theorem for Markov semigroup. Harris's Theorem originated in the paper \cite{H2} where Harris gave conditions for existence
and uniqueness of a steady state for Markov processes. It was then pushed forward by
Meyn and Tweedie in \cite{MT} to show exponential convergence to equilibrium. 
\cite{HM} gives an efficient way of getting quantitative rates for convergence to equilibrium
once   the assumptions have been quantitatively verified.  We give the precise statement in the next section.

\medskip

One advantage of the Harris method is that it directly yields convergence for a  wide range of initial conditions (there are many choice of $m$ and $H$ in Theorem \ref{T12}), while previous proofs of convergence to equilibrium mainly use some strongly weighted $L^2$ or $H^1$ norms (typically with a weight which is the inverse of a Gaussian). The Harris method  also gives existence of stationary solutions under  general conditions; in some cases these are explicit and easy to find, but in other cases such as the two models in our paper they can be nontrivial and non-explicit. Also the Harris method  provides a quantitative rate of convergence to the steady state, which is better than non-quantitative type argument such as Krein Rutman theorem. 
 \medskip
 
Here we briefly introduce the main idea of the paper. The paper uses Harris method to prove convergence. Roughly speaking, Harris method says that Lyapunov function plus positivity condition on a large ball implies $L^1(m)$ convergence for some weight function $m$. The Lyapunov function is easy to find so we mainly prove positivity on a large ball. The proof mainly contains three steps. First we prove that $f$ is above a constant on a point, then we use regularity method to prove the continuity of the solution, thus we can obtain that $f$ is above a constant in a small ball. Finally we use the spreading of positivity lemma which says if 
$f \ge \delta $ in $[0,t) \times B_r(x_0, v_0)$, then $f \ge \delta $ in $f \ge K\delta $ in $[\frac t 2, t) \times B_{\alpha r}(x_0, v_0)$ for any $\alpha>1$ and some $K>0$, we proved the positivity thus the convergence in $L^1(m)$ is proved. We use the Duhamel's formula and regularity estimate to prove convergence $p \in(1, \infty]$. This way of proving convergence for linear models is quite new and have the potential to be extended to other equations where exponential trend to equilibrium has not yet been shown.
 
\medskip

Let us end the introduction by describing the plan of the paper. In Section \ref{sec2}, we introduce Harris Theorem. 
In Section \ref{sec3} we present the proof of a regularization estimate on $S_\LL$. In Section \ref{sec4} we prove the convergence result in $L^1(m)$. In Section \ref{sec5} we prove the theorem in the case of $L^p(m)$ with general $p$. In section \ref{sec6}, we compute the Lyapunov function for the two equations. Finally in Appendix we prove some useful lemmas.

\bigskip
\section{Harris Theorem and existence of steady state}
\label{sec2}
\setcounter{equation}{0}
\setcounter{theo}{0}
In this section we introduce a PDE proof of Harris-Doeblin theorem and the existence of steady state by S. Mischler and J. A. Ca\~nizo.
\begin{theo}\label{T21}
(Harris-Doeblin Theorem)  We consider a Markov semigroup $S_\LL(t)$ with generator $\LL$ and define $S_t:=S_\LL(t)$, we assume that\\
(H1)(Lyapunov condition) There exists some weight function $m: \R^d \to [1, \infty)$ satisfying $m(x) \to \infty $ as $|x| \to \infty$ and there
exist some constants $\alpha > 0, b > 0$ such that
\beqn
\nonumber
\LL^* m \le - \alpha m + b,
\eeqn
(H2)(Harris condition) For any $R > 0$, there exist a constant $T = T(R)> 0$ and a positive, nonzero measure $\mu = \mu(R)$ such that
\beqn
\nonumber
S_T f \ge \mu \int_{B_R}
f, \quad \forall f \in L^1(m), \quad f \ge 0, \quad \Vert f \Vert_{L^1(m)} \le 1,
\eeqn
where  $B_{R}$ denotes the ball centered at origin with radius $R$.  Suppose the Markov semigroup $S_t$ on $L^1(m)$ which satisfies (H1)
and (H2). Then there exist some constants $C \ge 1$ and $a < 0$ such that
\beqn
\nonumber
\Vert S_t f \Vert_{L^1(m)} \le C e^{at} \Vert f \Vert_{L^1(m)},\quad  \forall t \ge 0,  \quad \forall f \in L^1(m),\quad   \MM( f)= 0.
\eeqn
\end{theo}

\begin{rem}
Usually the statement of Harris condition do not requires
\beqn
\nonumber
\Vert f \Vert_{L^1(m)} \le 1,
\eeqn
but in fact  conclusion remains the same since for every function $f$ we can consider $\frac 1 \lambda f$ and use the linearity of the solution. This additional assumption will be helpful in Section \ref{sec4}.
\end{rem}
\begin{rem}
In fact this version of Harris-Doeblin Theorem is a little stronger than the version in \cite{MT} because this version do not require a minimum of $T$ for all $R$, in this version it may happen that 
\beqn
\nonumber
T(R) \to 0,  \ \hbox{as} \ R \to \infty,
\eeqn
while in \cite{MT} they require a minimum $t_*>0$ for all $R>0$.
\end{rem}

Before proving the theorem, we first prove a useful lemma.

\begin{lem}\label{L21} (Doeblin's variant). Under assumption (H2), if $f \in  L^1(m)$, with $m(x) \to \infty $ as
$|x| \to \infty$, satisfies
\beqn\label{E21}
\Vert f \Vert _{L^1} \ge \frac {4} {m(R)}
\Vert f \Vert _{L^1(m)},  \quad \MM(f) = 0, \quad \Vert f \Vert_{L^1(m)} \le 1,
\eeqn
we then have
\beqn
\nonumber
\Vert S_T f \Vert _{L^1}  \le (1- \frac {\langle \mu(R) \rangle}{2})\Vert f \Vert _{L^1},
\eeqn
where
\beqn
\nonumber
\langle \mu \rangle = \int_{\R^d} \mu, \quad m(R) := \min \{|m(x)| , x \in B_R^c \},
\eeqn
\end{lem}

\begin{proof} From the hypothesis (\ref{E21}), we have
\bear
\nonumber
\int_{B_R} f_\pm &=& \int_{\R^d} f_\pm - \int_{B_R^c} f_\pm
\\ \nonumber
&\ge& \frac 1 2 \int_{\R^d} |f|- \frac {1} {m(R)} \int_{\R^d} |f| m  \ge \frac 1 4  \int_{\R^d} |f|,
\eear
since
\beqn
\nonumber
\Vert f_{\pm} \Vert_{L^1(m)} \le \Vert f \Vert_{L^1(m)}.
\eeqn
Together with (H2), we get
\beqn
\nonumber
S_T f_\pm \ge \frac {\mu(R)} 4 \int_{\R^d} |f| :=\eta,
\eeqn
We deduce
\beqn
\nonumber
|S_T f| \le |S_T f_+ -  \eta| + |S_T f_- -  \eta| = S_Tf_+ - \eta + S_Tf_- - \eta = S_T |f| - 2\eta,
\eeqn
and next
\beqn
\nonumber
\int_{\R^d} |S_T f| \le \int_{\R^d} S_T|f| - 2\eta = \int_{\R^d} (|f| - \frac {\mu(R)} 2 \int_{\R^d} |f|),
\eeqn
which is nothing but the announced estimate.
\end{proof}

Then we come to the proof of Theorem \ref{T21}.\\
\begin{proof} {\bf Proof of Theorem \ref{T21}}. We split the proof in several steps. In Step 1-6 we will assume $\Vert f_0 \Vert_{L^1(m)} \le 1$.\\
Step 1. We fix $f_0 \in L^1(m), \MM (f)   = 0$, and we denote $f_t := S_t f_0$. From (H1), we have
\beqn
\nonumber
\frac {d} {dt} \Vert f \Vert_{L^1(m)} \le -\alpha \Vert f_t \Vert_{L^1(m)} +b \Vert f_t \Vert_{L^1}, \quad \forall t \ge 0, 
\eeqn
from what we deduce
\beqn
\nonumber
\Vert S_T f \Vert_{L^1(m)} \le e^{-\alpha t} \Vert f_0 \Vert_{L^1(m)} +(1-e^{-\alpha t})\frac b \alpha  \Vert f_0 \Vert_{L^1}, \quad \forall t \ge 0, 
\eeqn
In other words, we have proved
\beqn\label{E22}
\Vert S_T f \Vert_{L^1(m)} \le \gamma \Vert f_0 \Vert_{L^1(m)} +K  \Vert f_0 \Vert_{L^1}, \quad \forall t \ge 0, 
\eeqn
with $\gamma \in (0, 1)$ and $K > 0$. We fix $R > 0$ large enough such that $\frac {8b} \alpha \le m(R)$, then take $T = T(R)$ and $\mu = \mu(R)$, define
\beqn
\nonumber
\gamma :=e^{-\alpha T} , \quad K:=(1-e^{-\alpha T})\frac b \alpha,
\eeqn
Then we have $K/A \le (1 - \gamma)/2$ with $A := m(R)/4$. We also recall that
\beqn\label{E23}
\Vert S_T f \Vert_{L^1} \le  \Vert f_0 \Vert_{L^1} , \quad \forall t \ge 0.
\eeqn
We define
\beqn
\nonumber
\Vert  f \Vert_{\beta} =  \Vert f_0 \Vert_{L^1}+ \beta \Vert f_0 \Vert_{L^1(m)},
\eeqn
and we observe that the following alternative holds
\beqn\label{E24}
\Vert  f_0 \Vert_{L^1(m)} \le A  \Vert f_0 \Vert_{L^1},
\eeqn
or
\beqn\label{E25}
\Vert  f_0 \Vert_{L^1(m)} > A  \Vert f_0 \Vert_{L^1}.
\eeqn
Step 2. By Lemma \ref{L21} that under condition (\ref{E24}), there holds
\beqn\label{E26}
\Vert  S_T f_0 \Vert_{L^1} \le \gamma_1  \Vert f_0 \Vert_{L^1}, \quad \gamma_1 \in (0,1 ),
\eeqn
and more precisely $\gamma_1 := 1 - \langle \mu \rangle/2$, which is nothing but the conclusion of Lemma \ref{L21}.\\
Step 3. We claim that under condition (\ref{E24}), there holds
\beqn\label{E27}
\Vert  S_T f_0 \Vert_{\beta} \le \gamma_2 \Vert f_0 \Vert_{\beta} , \quad \gamma_2 := \max(\frac {\gamma_1+1} {2},\gamma),
\eeqn
for $\beta > 0$ small enough. Indeed, using (\ref{E22}) and (\ref{E27}), we compute
\bear
\nonumber
\Vert  S_Tf_0 \Vert_{\beta} &=&   \Vert S_Tf_0 \Vert_{L^1} + \beta \Vert S_Tf_0 \Vert_{L^1(m)}
\\ \nonumber
&\le& (\gamma_1 +K\beta)  \Vert f_0 \Vert_{L^1} +\gamma \beta  \Vert f_0 \Vert_{L^1(m)},
\eear
and we take $\beta > 0$ such that $\gamma_1 + K\beta \le \gamma_2$.\\
Step 4. We claim that under condition (\ref{E25}), there holds
\beqn\label{E28}
\Vert  S_T f_0 \Vert_{L^1(m)} \le \gamma_3  \Vert f_0 \Vert_{L^1(m)}, \quad \gamma_3 := \frac {\gamma+1} {2}.
\eeqn
Indeed we compute
\beqn
\nonumber
\Vert  S_T f_0 \Vert_{L^1(m)} \le \gamma  \Vert f_0 \Vert_{L^1(m)} + \frac K A \Vert f_0 \Vert_{L^1(m)} = \gamma_3 \Vert f_0 \Vert_{L^1(m)}.
\eeqn
Step 5. We claim that under condition (\ref{E25}), there holds
\beqn\label{E29}
\Vert  S_T f_0 \Vert_{\beta} \le \gamma_4  \Vert f_0 \Vert_{\beta}, \quad \gamma_4 := \frac {\gamma_3+1/\beta} {1+1/\beta}.
\eeqn
Indeed, using (\ref{E23}) and (\ref{E28}), we compute
\bear
\nonumber
\Vert  S_Tf_0 \Vert_{\beta} &=&   \Vert S_Tf_0 \Vert_{L^1} + \beta \Vert S_Tf_0 \Vert_{L^1(m)}
\\ \nonumber
&\le&  \Vert f_0 \Vert_{L^1} +\gamma_3 \beta  \Vert f_0 \Vert_{L^1(m)}
\\ \nonumber
&\le& (1-\epsilon) \Vert f_0 \Vert_{L^1} + (\epsilon + \gamma_3 \beta  )\Vert f_0 \Vert_{L^1(m)},
\eear
and we choose $\epsilon \in (0, 1)$ such that $1 - \epsilon = \epsilon/\beta + \gamma_3$.\\
Step 6. By gathering (\ref{E27}) and (\ref{E29}), we see that we have
\beqn
\nonumber
\Vert  S_T f_0 \Vert_{\beta} \le \gamma_5  \Vert f_0 \Vert_{\beta}, \quad \gamma_5 := \max(\gamma_2, \gamma_4) \in(0,1),
\eeqn
for some well chosen $\beta > 0$. By iteration, we get
\beqn
\nonumber
\Vert  S_{nT} f_0 \Vert_{\beta} \le \gamma_5^n  \Vert f_0 \Vert_{\beta}, 
\eeqn
and we then conclude there exist some constants $C \ge 1$ and $a < 0$ such that
\beqn
\nonumber
\Vert S_t f \Vert_{L^1(m)} \le C e^{at} \Vert f \Vert_{L^1(m)},\quad  \forall t \ge 0,  \quad \forall f \in L^1(m), \Vert  f \Vert_{L^1(m)} \le 1, \quad   \MM( f)= 0.
\eeqn 
Step 7. (Linearity argument) For general $f$, we can always find $\lambda>0$ such that $\Vert \frac 1 \lambda f \Vert_{L^1(m)} \le 1 $, since $S_t$ is linear we have
\beqn
\nonumber
 \Vert S_t f \Vert_{L^1(m)}  = \lambda \Vert S_t( \frac 1 \lambda f )\Vert_{L^1(m)} \le \lambda C e^{at} \Vert \frac 1 \lambda f \Vert_{L^1(m)} =  C e^{at} \Vert  f \Vert_{L^1(m)} ,\quad  \forall t \ge 0, 
\eeqn 
for all  $f \in L^1(m), \quad   \MM( f)= 0.$
\end{proof}

The Lyapunov condition also provides a sufficient condition for the existence of an invariant measure (for the dual semigroup).
\begin{theo}\label{T22}
Any mass conserving positive Markov semigroup ($S_t$) which fulfills the above Lyapunov
condition has at least one invariant borelian measure $G \in M^1(m)$, where $M^1$ is the space of measures.
\end{theo}
\begin{proof} 
Step 1. We prove that ($S_t$) is a bounded semigroup. For $f_0 \in M^1(m)$, we define $f_t := S_\LL(t)f_0$, and we easily compute
\beqn
\nonumber
\frac d {dt} \int |f_t| m \le \int |f_t|\LL^* m \le\int |f_t| (-am + b).
\eeqn
Using the mass conservation and positivity, integrating the above differential inequality, we get
\bear
\nonumber
\int |f_t| m &\le& e^{-at} \int |f_0| m + \frac {b} a (1 - e^{-at}) \int |f_0|
\\ \nonumber
&\le& max (1, \frac b a) \int |f_0| m, \forall t \ge 0,
\eear
so that $(S_t)$ is bounded in $M^1(m)$.\\
Step 2. We prove the existence of a steady state, more precisely, we start proving that there exists a positive and normalized steady state $G \in M^1(m)$. For the equivalent norm $||| \cdot |||$ defined on $M^1(m)$ by
\bear
\nonumber
|||f||| := \sup_{t>0} \Vert
S_\LL(t) f\Vert_{M^1(m)},
\eear
we have $|||S_\LL(t)f||| \le |||f|||$ for all $t \ge 0$, that is the semigroup $S_\LL$ is a contraction semigroup
on $(M^1(m), ||| \cdot |||)$. There exists $R > 0$ large enough such that the intersection of the closed
hyperplane $ \{ f \in M^1(m); \MM(f)  = 1\}$ and the closed ball of radius $R$ in $(M^1(m), ||| \cdot |||)$ is a
convex, non-empty subset. Then consider the closed, weakly * compact convex set
\bear
\nonumber
\mathbb{K} := \{ f \in M^1(m) ; |||f||| \le R, f \ge 0, \MM( f) = 1 \},
\eear
Since $S_\LL(t)$ is a linear, weakly * continuous, contraction in $(M^1(m), ||| \cdot |||)$ and $\MM( S_\LL(t)  f)  = \MM (f) $ for all $t \ge 0$, we see that $\mathbb {K}$ is stable under the action of the semigroup. Therefore we apply the Markov-Kakutani fixed point theorem and we conclude that there exists $G \in \mathbb{K}$ such that $S_\LL(t)G = G$ for all $t \ge 0$. Therefore we have in particular $G \in D(\LL)$ and $\LL G = 0$. 
\end{proof}

\bigskip
\section{Regularization property of  $S_\LL$}
\label{sec3}
\setcounter{equation}{0}
\setcounter{theo}{0}

The aim of this section is to establish the following regularization property. The proof closely follows the proof of similar results in \cite{H, MM, V}. In the whole section, $m$ and $H$ refers to the one defined in Theorem \ref{T12}.
\begin{theo}\label{T31}
 Consider the weight function $m$  as defined in Theorem \ref{T12} satisfies condition (C1)-(C4), then there exist  $\eta, C  > 0$ such that
\beqn
\nonumber
\Vert S_\LL(t)f \Vert_{L^2(m)} \le  \frac C {t^{\frac {5d+2} {4} } } \Vert f \Vert_{L^1(m)}, \ \ \forall t \in [0,\eta] .
\eeqn
for some weight function $m$. In addition, for any integer $k>0$ there exist some $\alpha(k), C(k)>0$ such that
\beqn
\nonumber
\Vert S_\LL(t)f \Vert_{H^k(m)} \le  \frac C {t^{\alpha} } \Vert f \Vert_{L^1(m)}, \ \ \forall t \in [0,\eta] .
\eeqn
as a consequence we have 
\beqn
\nonumber
\Vert S_\LL(t)f \Vert_{C^{2,\delta}} \le  \frac C {t^{\zeta} } \Vert f \Vert_{L^1(m)}, \ \ \forall t \in [0,\eta] ,
\eeqn
for some $\delta \in (0, 1) , \zeta>0$.
\end{theo}

We start with some elementary lemmas.
\sk

\begin{lem}\label{L31}
For $f_t =S_\LL(t) f_0$,  define an energy functional
\bear\label{E31}
\nonumber
\FF(t, f_t ) &:= &A \Vert f_t \Vert_{L^2(m)}^2 + at\Vert \nabla_v f_t \Vert_{L^2(m)}^2 
\\ 
&+& 2 c t^2( \nabla_v f_t, \nabla_x f_t )_{L^2(m)} + bt^3 \Vert \nabla_x f_t \Vert_{L^2(m)}^2,
\eear
with $a, b, c >0, c \le \sqrt{ab} $ and $A$ large enough. Then there exist $\eta>0$ such that
\bear
\nonumber
\frac {d } {dt}\FF(t, f_t )\le -L(\Vert \nabla_v f_t \Vert_{L^2(m)}^2 + t^2 \Vert \nabla_x f _t \Vert_{L^2 ( m) }^2 ) +C \Vert f_t \Vert_{L^2(m)}^2 ,
\eear
for all $t \in [0, \eta]$ and some $L>0, C>0$, as a consequence, we have
\beqn
\nonumber
\Vert S_\LL f_0 \Vert_{H^1(m)} \le Ct^{-\frac 3 2} \Vert f_0 \Vert_{L^2(m)},
\eeqn
for all $t \in [0, \eta]$, iterating $k$ times we get
\beqn
\nonumber
\Vert S_\LL f_0 \Vert_{H^k(m)} \le Ct^{-\frac {3k} 2 } \Vert f_0 \Vert_{L^2(m)}.
\eeqn
\end{lem}

\begin{rem}\label{R41}
We need to note here that if we consider 
\bear
\nonumber
\FF^*(t, f_t ) &:= &A \Vert f_t \Vert_{L^2(m)}^2 + at^2\Vert \nabla_v f_t \Vert_{L^2(m)}^2 
\\ \nonumber
&+& 2 c t^4( \nabla_v f_t, \nabla_x f_t )_{L^2(m)} + bt^6 \Vert \nabla_x f_t \Vert_{L^2(m)}^2,
\eear
then by the same proof we have
\bear
\nonumber
\frac {d } {dt}\FF^*(t, f_t )\le -L(\Vert \nabla_v f_t \Vert_{L^2(m)}^2 + t^4 \Vert \nabla_x f _t \Vert_{L^2 ( m) }^2 ) +C \Vert f_t \Vert_{L^2(m)}^2 ,
\eear
for all $t \in [0, \eta]$, for some $L>0, C>0$. This version will be useful in the later proof.
\end{rem}

\begin{proof}
We only prove the case $k=1$, for $k = 2$, one need only replace $f$ by $\partial_{x_i}f$ and $\partial_{v_i} f$, similarly for $k >2$.
\noindent  
First by Theorem \ref{T12} and Remark \ref{R11} we have
\beqn
\nonumber
(f, \LL g )_{L^2(m)}+(g, \LL f)_{L^2(m)} = -2K(\nabla_v f, \nabla_v g)_{L^2(m)} + (f, g \phi_2(m))_{L^2(m)},
\eeqn
for any $f, g \in L^2(m) $, without loss of generality we will assume $K=1$. By condition (C2), we have
\beqn
\nonumber
\frac d {dt} \Vert f \Vert_{L^2 (m)}^2  = (f, \LL f )_{L^2(m)} \le - \Vert \nabla_v f \Vert_{L^2(m)}^2 - C_1 \Vert f \Vert_{L^2(m H^{1/2})}^2 +C_2\Vert f \Vert_{L^2(m)}^2.
\eeqn
We compute
\bear\label{E32}
\partial_{x_i}\LL f &=&  \LL \partial_{x_i} f  +\partial_{x_i }  \Phi(x) \cdot \nabla_x f + \partial_{x_i}B(x, v) \cdot \nabla_v f 
\\ \nonumber
&&+ \partial_{x_i} \textup{div}_x\Phi(x) f+ \partial_{x_i} \textup{div}_v B(x, v) f,
\eear 
by condition (C3)
\beqn
\nonumber
|\partial_{x_i }  \Phi(x) | +|\partial_{x_i}B(x, v)| +| \partial_{x_i} \textup{div}_x\Phi(x) |+| \partial_{x_i} \textup{div}_v B(x, v)  | \le \epsilon H + C ,
\eeqn
for some $C>0$, we have
\bear
\nonumber
&&\frac d {dt} \Vert \partial_{x_i} f \Vert_{L^2 (m)}^2
\\ \nonumber 
&=& (\partial_{x_i} f, \LL \partial_{x_i} f )_{L^2(m)}  + (\partial_{x_i} f, \partial_{x_i }  \Phi(x) \cdot \nabla_x f + \partial_{x_i}B(x, v) \cdot \nabla_v f )_{L^2(m)} 
\\ \nonumber
&&+ (\partial_{x_i} f,  \partial_{x_i} \textup{div}_x\Phi(x) f+ \partial_{x_i} \textup{div}_v B(x, v) f  )_{L^2(m)} 
\\ \nonumber
&\le& - \Vert \nabla_v (\partial_{x_i}f) \Vert_{L^2(m)}^2 - C_1 \Vert \partial_{x_i} f \Vert_{L^2(m H^{1/2})}^2 +C_2 \Vert  \partial_{x_i} f \Vert_{L^2(m)}^2 
\\ \nonumber
&&+ \epsilon (\Vert \nabla_{v} f\Vert_{L^2(m H^{1/2})}^2 + \Vert \nabla_{x} f\Vert_{L^2(m H^{1/2})}^2 +\Vert  f\Vert_{L^2(m H^{1/2})}^2)
\\ \nonumber
&&+ C (\Vert \nabla_{v} f\Vert_{L^2(m)}^2 + \Vert \nabla_{x} f\Vert_{L^2(m)}^2 +\Vert  f\Vert_{L^2(m)}^2).
\eear
Summing over $i=1, 2, 3, ..., n$ , we get
\bear
\nonumber
&&\frac d {dt} \Vert \nabla_{x} f \Vert_{L^2 (m)}^2
\\ \nonumber
&\le& - \sum_{i=1}^n \Vert \nabla_v ( \partial_{x_i} f)\Vert_{L^2(m)}^2 - \frac {C_1} 2 \Vert \nabla_x  f \Vert_{L^2(m H^{1/2})}^2+C\Vert  \nabla_x f \Vert_{L^2(m)}^2
\\ \nonumber
&& + C \Vert \nabla_{v} f \Vert_{L^2(m H^{1/2})}^2+C\Vert  f\Vert_{L^2(m H^{1/2})}^2,
\eear
for some $C > 0$. Similarly using
\beqn\label{E33}
\partial_{v_i}\LL f=\LL\partial_{v_i} f -\partial_{x_i}f +\partial_{v_i}B(x, v) \cdot \nabla_v f + \partial_{v_i} \textup{div}_v B(x, v) f,
\eeqn
and since
\beqn
\nonumber
|\partial_{v_i}B(x, v)| +| \partial_{v_i} \textup{div}_v B(x, v)  | \le \epsilon H + C,
\eeqn
we have
\bear
\nonumber
&&\frac d {dt} \Vert \partial_{v_i} f \Vert_{L^2 (m)}^2
\\ \nonumber &=& (\partial_{v_i} f, \LL \partial_{v_i} f )_{L^2(m)}  - (\partial_{x_i} f, \partial_{v_i} f )_{L^2(m)} +(\partial_{v_i} f, \partial_{v_i}B(x, v) \cdot \nabla_v f )_{L^2(m)} 
\\ \nonumber
&&+(\partial_{v_i} f, \partial_{v_i} \textup{div}_v B(x, v) f )_{L^2(m)} 
\\ \nonumber
&\le& - \Vert \nabla_v (\partial_{v_i}f) \Vert_{L^2(m)}^2 - C_1 \Vert \partial_{v_i} f \Vert_{L^2(m H^{1/2})}^2 +C_2\Vert  \partial_{v_i} f \Vert_{L^2(m)}^2 +\epsilon \Vert \nabla_{v} f \Vert_{L^2(m H^{1/2})}^2
\\ \nonumber
&&+C \Vert \nabla_{v} f \Vert_{L^2(m)}^2-(\partial_{x_i} f, \partial_{v_i} f )_{L^2(m)}+ C \Vert  f \Vert_{L^2(m H^{1/2})}^2.
\eear
Summing over $i=1, 2, ..., n$ we get
\bear
\nonumber
&&\frac d {dt} \Vert \nabla_{v} f \Vert_{L^2 (m)}^2
\\ \nonumber
&\le& - \sum_{i=1}^n \Vert \nabla_v (\partial_{v_i}f) \Vert_{L^2(m)}^2 -  \frac {C_1} {2} \Vert \nabla_{v} f \Vert_{L^2(m H^{1/2})}^2+ C \Vert  f \Vert_{L^2(m H^{1/2})}^2.
\\ \nonumber
&&+C \Vert \nabla_{v} f \Vert_{L^2(m)}^2 - (\nabla_v f , \nabla_x f )_{L^2(m)}.
\eear
For the crossing term, using (\ref{E32}), (\ref{E33}) and condition (C2) and (C3), we have
\bear
\nonumber
&&\frac d {dt} 2(\partial_{v_i} f, \partial_{x_i} f )_{L^2 (m)}  
\\ \nonumber 
&=& (\partial_{v_i} f, \LL \partial_{x_i} f )_{L^2(m)}   + (\partial_{v_i} f, \partial_{x_i }  \Phi(x) \cdot \nabla_x f + \partial_{x_i}B(x, v) \cdot \nabla_v f )_{L^2(m)} 
\\ \nonumber
&&+ (\partial_{v_i} f,  \partial_{x_i} \textup{div}_x\Phi(x) f+ \partial_{x_i} \textup{div}_v B(x, v) f  )_{L^2(m)} 
\\ \nonumber 
&&+ (\partial_{x_i} f, \LL \partial_{v_i} f )_{L^2(m)} - (\partial_{x_i} f, \partial_{x_i} f )_{L^2(m)} +(\partial_{x_i} f, \partial_{v_i}B(x, v) \cdot \nabla_v f )_{L^2(m)} 
\\ \nonumber
&&+(\partial_{x_i} f, \partial_{v_i} \textup{div}_v B(x, v) f )_{L^2(m)},
\eear
We split into two parts, for the first part we compute
\bear
\nonumber
&&(\partial_{v_i} f, \LL \partial_{x_i} f )_{L^2(m)} + (\partial_{x_i} f, \LL \partial_{v_i} f )_{L^2(m)} - \Vert \partial_{x_i} f \Vert_{L^2(m )}^2
\\ \nonumber
&=&-2 (\nabla_v (\partial_{x_i} f), \nabla( \partial_{v_i} f) )_{L^2(m)}  +(\partial_{x_i} f, \phi_2(m) \partial_{v_i} f )_{L^2(m)} - \Vert \partial_{x_i} f \Vert_{L^2(m )}^2
\\ \nonumber
&\le& -2 (\nabla_v (\partial_{x_i} f), \nabla( \partial_{v_i} f) )_{L^2(m)}  - \Vert \partial_{x_i} f \Vert_{L^2(m )}^2 +C ( |\nabla_{v} f|, |\nabla_x f| )_{L^2(m H^{1/2})},
\eear
for the second part we have
\bear
\nonumber
&&  (\partial_{v_i} f, \partial_{x_i }  \Phi(x) \cdot \nabla_x f + \partial_{x_i}B(x, v) \cdot \nabla_v f )_{L^2(m)} +(\partial_{x_i} f, \partial_{v_i}B(x, v) \cdot \nabla_v f )_{L^2(m)}
\\ \nonumber 
&&+ (\partial_{v_i} f,  \partial_{x_i} \textup{div}_x\Phi(x) f+ \partial_{x_i} \textup{div}_v B(x, v) f  )_{L^2(m)}   +(\partial_{x_i} f, \partial_{v_i} \textup{div}_v B(x, v) f )_{L^2(m)} 
\\ \nonumber
&\le&C \Vert \nabla_{v} f \Vert_{L^2(m H^{1/2})}^2+C ( |\nabla_{v} f|, |\nabla_x f| )_{L^2(m H^{1/2})}  +C ( | f|, |\nabla_x f| )_{L^2(m H^{1/2})}
\\ \nonumber
&&+C  \Vert f \Vert_{L^2(m H^{1/2})}^2.
\eear
Gathering the two terms, and summing over $i$ we get
\bear
\nonumber
&&\frac d {dt} 2(\nabla_{v} f, \nabla_{x} f )_{L^2 (m)}  
\\ \nonumber
&\le&  -2 \sum_{i=1}^n (\nabla_v (\partial_{x_i} f), \nabla( \partial_{v_i} f) )_{L^2(m)}  - \Vert \nabla_{x} f \Vert_{L^2(m )}^2 +C \Vert \nabla_{v} f \Vert_{L^2(m H^{1/2})}^2
\\ \nonumber
&+& C ( |\nabla_{v} f|, |\nabla_x f| )_{L^2(m H^{1/2})} + C ( | f|, |\nabla_x f| )_{L^2(m H^{1/2})} + C \Vert  f \Vert_{L^2(m H^{1/2})}^2.
\eear
For the very definition of $\FF$ in (\ref{E31}), we easily compute
\bear
\nonumber
\frac d {dt} \FF(t, f_t) &=& A \frac d {dt} \Vert f_t \Vert_{L^2(m)}^2 + at   \frac d {dt} \Vert \nabla_v f_t \Vert_{L^2(m)}^2 +2 c t^2 \frac d {dt} ( \nabla_v f_t, \nabla_x f_t)_{L^2(m)} 
\\ \nonumber
&&+ bt^3 \frac d {dt} \Vert \nabla_x f_t \Vert_{L^2(m)}^2+  a  \Vert \nabla_v f_t \Vert_{L^2(m)}^2 +  4c t ( \nabla_v f_t, \nabla_x f_t )_{L^2(m)} 
\\ \nonumber
&&+ 3 b t^2 \Vert \nabla_x f_t \Vert_{L^2(m)}^2.
\eear
Gathering all the inequalities above together, we have
\beqn
\nonumber
\frac d {dt} \FF(t, f_t) \le T_1 + T_2 + T_3,
\eeqn
with 
\bear
\nonumber
T_1&=&  (a- A+ Cat )\Vert \nabla_v f_t \Vert_{L^2(m)}^2+ ( 3bt^2-   c  t^2+C b t^3)\Vert \nabla_x f_t \Vert_{L^2(m)}^2
\\ \nonumber
&+& (4c t - at ) ( \nabla_{v} f_t , \nabla_{x} f_t )_{L^2(m)}+C_2A\Vert  f_t \Vert_{L^2(m)}^2
\\ \nonumber 
&\le& -L(\Vert \nabla_v f_t \Vert_{L^2(m)}^2 + t^2 \Vert \nabla_x f _t \Vert_{L^2 ( m) }^2) +C\Vert f_t \Vert_{L^2(m)}^2,
\eear
for some $L, C  > 0$, if $c >6b$, $A \gg a, b, c$  and $0 < \eta$ small. For the term $T_2$ we have
\bear
\nonumber
T_2 &=&\sum_{i=1}^d [ -a t \Vert  \nabla_{v}  ( \partial_{v_i} f_t ) \Vert_{L^2(m)}^2 -b t^3 \Vert \nabla_{v}  ( \partial_{x_i}  f_t  )\Vert_{L^2(m)}^2 
\\ \nonumber
&&- 2 c t^2 (\nabla_v ( \partial_{x_i}f_t  ) , \nabla_v ( \partial_{v_i}f_t  ) )_{L^2(m)} ] \le 0,
\eear
since
\bear
\nonumber
&&|2 c t^2 (\nabla_v ( \partial_{x_i}f_t  ) , \nabla_v ( \partial_{v_i}f_t ) )_{L^2(m)}  |
\\ \nonumber
&\le& a t \Vert  \nabla_{v}  ( \partial_{v_i} f_t ) \Vert_{L^2(m)}^2 +b t^3  \Vert  \nabla_{v}  ( \partial_{v_i} f_t ) \Vert_{L^2(m)}^2,
\eear
by our choice on $a, b, c$. For the term $T_3$
\bear
\nonumber
T_3 &=& - \frac {C_1} 2  b t^3 \Vert \nabla_x  f_t \Vert_{L^2(m H^{1/2})}^2 + \left(-\frac {C_1} 2 at + C bt^3 +C c t^2\right)  \Vert \nabla_{v} f_t \Vert_{L^2(m H^{1/2})}^2
\\ \nonumber
&& + Cc  t^2  ( |\nabla_{v} f_t|, |\nabla_x f_t| )_{L^2(m H^{1/2})} + C c t^2 ( | f|, |\nabla_x f| )_{L^2(m H^{1/2})} 
\\ \nonumber
&& + (-C_1 A +C b t^3+C a t+C c t^2 )\Vert  f_t \Vert_{L^2(m H^{1/2})}^2 \le 0,
\eear
by taking $A \gg a, b, c$, $ab \gg c^2$. So by taking $A $ large and $0 < \eta$ small ($t \in [0, \eta]$), we conclude to
\bear
\nonumber
\frac d {dt} \FF(t, f_t)\le -L(\Vert \nabla_v f_t \Vert_{L^2(m)}^2 + t^2 \Vert \nabla_x f _t \Vert_{L^2 ( m) }^2) +C\Vert f_t \Vert_{L^2(m)}^2,
\eear
for some $L, C  > 0$, and that ends the proof.
\end{proof}

\begin{lem}\label{L32} We have
\bear
\nonumber
\Vert \nabla_{x, v} (f_tm) \Vert_{L^2}^2 \le\Vert \nabla_{x, v} f_t \Vert_{L^2(m)}^2+C \Vert f_t \Vert_{L^2(m)}^2,
\eear
for some constant $C$.
\end{lem}
\begin{proof}
\noindent
We have
\bear
\nonumber
\Vert \nabla_{x, v} (f_tm) \Vert_{L^2}^2 &=& \Vert m  \nabla_{x, v} f_t \Vert_{L^2}^2+\Vert f_t\nabla_{x, v} m \Vert_{L^2}^2 + 2(f_t\nabla_{x, v}m,m\nabla_{x, v}f_t )_{L^2}
\\ \nonumber
&=& \Vert  \nabla_{x, v} f_t \Vert_{L^2(m)}^2+\Vert f_t\nabla_{x, v} m \Vert_{L^2}^2 -\frac 1 2(f_t^2,\Delta_{x, v}(m^2) )_{L^2}
\\ \nonumber
&=&  \Vert  \nabla_{x, v} f_t \Vert_{L^2(m)}^2 +(f_t^2, |\nabla_{x, v}m|^2-\frac 1 2\Delta_{x, v}(m^2) )_{L^2}
\\ \nonumber
&=& \Vert  \nabla_{x, v} f_t \Vert_{L^2(m)}^2 - (f_t^2, m\Delta_{x, v}m )_{L^2},
\eear
by condition (C4)
\beqn
\nonumber
\frac { \Delta_{x, v} m}{m}    \ge C,
\eeqn
for some constant $C$, we are done. 
\end{proof}

\begin{lem}\label{L33}
Nash's inequality: for any $f \in L^1(\R^d) \cap H^{1}(\R^d) $,there exist a constant $C_d$ such that:
\beqn
\nonumber
\Vert f \Vert_{L^2}^{1+\frac 2 d} \le C_d \Vert f \Vert_{L^1}^{\frac  2 d}\Vert \nabla_v f \Vert_{L^2}.
\eeqn
 \end{lem} 
\noindent For the proof  we refer to \cite{LL}, Section 8.13 for instance.
\qed

\begin{lem}\label{L34} There exist $\lambda>0$ such that
\bear\label{E34}
\frac d {dt } \Vert f \Vert_{L^1(m)}\le \lambda   \Vert f \Vert_{L^1(m)},
\eear
which implies
\bear
\nonumber
\Vert f_t \Vert_{L^1(m)} \le Ce^{\lambda t} \Vert f_0\Vert_{L^1(m)}.
\eear
In particular we have
\bear\label{E35}
\Vert f_t \Vert_{L^1(m)} \le C \Vert f_0\Vert_{L^1(m)},\quad \forall t \in [0,\eta],
\eear
for some constant $C>0$.
\end{lem}
\begin{proof}
It' s an immediate consequence of the Lyapunov condition (C1).
\end{proof}

\sk

Now we come to the proof of Theorem \ref{T31}.

\sk

\begin{proof}
\noindent ({\bf Proof of Theorem \ref{T31}.}) We define
\beqn
\nonumber
\GG(t, f_t)=B \Vert  f_t \Vert_{L^1(m)}^2 + t^Z \FF^* (t, f_t),
\eeqn
with $B, Z > 0 $ to be fixed and $\FF^*$ defined in Remark \ref{R41}. We choose $t \in [0, \eta],$  $\eta$ small enough such that $(a+b+c) Z\eta^{Z+1} \le \frac 1 2 L \eta^Z$ ($a, b, c, L$ are also defined Remark \ref{R41}). By  (\ref{E34}) and Remark \ref{R41}, we have
\bear
\nonumber
\frac  d {dt} \GG(t, f_t) &\le& \lambda B \Vert f_t \Vert_{L^1(m)}^2 +Zt^{Z-1} \FF^*(t, f_t) 
\\ \nonumber
&&- L t^Z(\Vert \nabla_v f_t \Vert_{L^2(m)}^2 + t^4 \Vert \nabla_x f _t \Vert_{L^2 ( m ) }^2 )+Ct^Z \Vert  f_t \Vert_{L^2(m)}^2 
\\ \nonumber
&\le&  \lambda B \Vert f_t \Vert_{L^1(m ) }^2  +Ct^{Z-1}\Vert  f_t \Vert_{L^2(m)}^2 
\\ \nonumber
&&- \frac L 2 t^Z(\Vert \nabla_v f_t \Vert_{L^2(m)}^2 + t^4 \Vert \nabla_x f _t \Vert_{L^2 ( m) }^2 ),
\eear
where $\lambda$ is defined in Lemma \ref{L34}. Nash's inequality and Lemma \ref{L32} imply
\bear
\nonumber
\Vert f_tm \Vert_{L^2} \le  C\Vert f m \Vert_{L^1}^{\frac 2 {d+2}} \Vert \nabla_{x, v}(f_tm)\Vert_{L^2}^{\frac d {d+2}}  \le  C\Vert f_t m \Vert_{L^1}^{\frac 2 {d+2}} (\Vert \nabla_{x, v} f m\Vert_{L^2} +C\Vert f_t m\Vert_{L^2} )^{\frac d {d+2}}.
\eear
Using Young's inequality, we have
\beqn
\nonumber
\Vert f_t \Vert_{L^2( m )}^2 \le C_{\epsilon} t^{-\frac 5 2 d} \Vert f \Vert_{L^1(m )}^2 + \epsilon t^5 (\Vert \nabla_{x, v} f_t \Vert_{L^2( m )}^2 + C\Vert  f_t \Vert_{L^2(m)}^2).
\eeqn
Taking $\epsilon$ small such that $C \epsilon \eta^3 \le  \frac 1 2$, we deduce
\beqn
\nonumber
\Vert f_t \Vert_{L^2(m)}^2 \le 2C_{\epsilon} t^{-\frac 5 2 d} \Vert f \Vert_{L^1(m)}^2 + 2\epsilon t^5 \Vert \nabla_{x, v} f_t \Vert_{L^2(m)}^2 .
\eeqn
Taking $\epsilon$ small we have
\beqn
\nonumber
\frac d {dt} \GG(t, f_t) \le \lambda B \Vert f_t \Vert_{L^1(m)}^2+ C_1t^{Z-1-\frac 5 2d} \Vert f_t \Vert_{L^1(m)}^2,
\eeqn
for some $C_1>0$. Choosing $Z=1+\frac 5 2 d$, and using (\ref{E35}), we deduce
\beqn
\nonumber
\forall t \in [0, \eta], \ \ \ \GG(t, f_t) \le \GG(0, f_0) + C_2\Vert f_0 \Vert_{L^1(m )}^2 \le C_3 \Vert f_0 \Vert_{L^1(m)}^2,
\eeqn
which proves
\beqn
\nonumber
\Vert S_\LL(t)f \Vert_{L^2(m)} \le  \frac C {t^{\frac {5d+2} {4} } } \Vert f \Vert_{L^1(m)}, \ \ \forall t \in [0,\eta] .
\eeqn
together with Lemma \ref{L31} ends the proof.
\end{proof}

\bigskip
\section{Convergence in $L^1(m)$}
\label{sec4}
\setcounter{equation}{0}
\setcounter{theo}{0}

In this section we prove the Harris condition (H2) for Theorem \ref{T12}, which would imply the convergence for $p=1$. Before the proof of the theorem, we first prove a useful lemma.
\begin{lem}\label{L41}
For any $R>0$, there exist $\gamma, \rho > 0 $ such that for any $ t, R > 0$, there exists $(x_0, v_0) \in B_\rho$ such that
\beqn
\nonumber
f(t, x_0, v_0) \ge \gamma \int_{B_R} f_0 .
\eeqn
$\gamma, \rho$ does not depend on $f_0, t $, while $x_0, v_0$ may depend on $f_0, t$
\end{lem}
\begin{proof}
From conservation of mass, we classically show that
 \beqn
\nonumber
\frac {d} {dt} \int_{\R^d} f(t, x, v )dx dv = 0,
\eeqn
so we have
\beqn\label{E41}
\Vert S_\LL(t) \Vert_{L^1 \to L^1} \le 1, \quad \forall t \ge 0,
\eeqn
Define the splitting of the operator $\LL$ by
\beqn
\nonumber
\BB =\LL-\AA, \quad \AA=M\chi_R(x, v),
\eeqn
with $M, R >0$ large, where $\chi$ is the cut-off function  such that  $\chi(x, v) \in [0, 1]$, $\chi(x, v) \in C^\infty $, $\chi(x, v) =1$ when $x^2+v^2 \le 1$ , $\chi(x, v) =0$ when $x^2+v^2 \ge 2$, and $\chi_R = \chi(x/R, v/R)$. By the Lyapunov function condition (H1) and taking $M, R$ large, we have
\beqn\label{E42}
\Vert S_\BB(t) \Vert_{L^1(m) \to L^1(m)} \le C e^{-\lambda t}, \quad \forall t \ge 0.
\eeqn
By Duhamel's formula
\beqn
\nonumber
S_\LL = S_\BB + S_\BB *\AA S_\LL,
\eeqn
we directly deduce from (\ref{E41}) and \ref{E42} that
\beqn
\nonumber
\Vert S_\LL(t) \Vert_{L^1(m) \to L^1(m)} \le A, \quad \forall t \ge 0,
\eeqn
for some $A>0$. We fix $R>0$ and take $g_0 =f_0 \1_{B_R} \in L^1(\R^d)$ such that $\hbox{supp} \ g_0 \subset B_R$, ,denote $g_t = S_\LL g_0, f_t = S_\LL f_0$, then we have
\beqn
\nonumber
\int_{\R^d} g_t  =\int_{\R^d} g_0  =\int_{B_R} g_0 = \int_{B_R} f_0.
\eeqn
Define
\beqn
\nonumber
m_1(R) = \max \{|m(x)|, x \in B_R    \}, \quad  m_2(R) = \min \{|m(x)|, x \in B_R^c    \},
\eeqn
We can see both $m_1, m_2 \to \infty$ as $R \to \infty$, moreover, since there exists $A > 0$ such that
\beqn
\nonumber
\int_{\R^d} g_t m   \le A \int_{\R^d} g_0 m   \le Am_1(R)\int_{B_R} g_0.
\eeqn
For any $\rho>0$, we write
\bear
\nonumber
\int_{B_\rho} g_t &=& \int_{\R^d } g_t - \int_{B_{\rho}^c} g_t
\\ \nonumber
&\ge& \int_{\R^d}g_0 - \frac 1 {m_2(\rho) }\int_{\R^d}g_t m 
\\ \nonumber
&\ge&  \int_{\R^d}g_0 - \frac {Am_1(R)} {m_2(\rho)} \int_{B_R}g_0 \ge \frac 1 2 \int_{B_R}g_0,
\eear
by taking $m_2(\rho) = 2Am(R)$. As a consequence, for any $t >0$, there exist a $(x_0, v_0) \in B_\rho$ which may depend on $g_0, t$ such that
\beqn
\nonumber
g(t, x_0,v_0) \ge\frac {1} {|B_\rho|} \int_{B_\rho}g_t  := \gamma  \int_{B_R} g_0.
\eeqn
By the maximum principle we have
\beqn
\nonumber
f(t, x_0, v_0) \ge g(t, x_0,v_0) \ge \gamma  \int_{B_R} g_0 = \gamma \int_{B_R} f_0.
\eeqn
\end{proof}
Before coming to the final proof we still need a theorem on spreading of positivity. Define
\beqn
\nonumber
\bar{B}_r(x_0, v_0) = \{ (x, v) \in \R^d \times \R^d : |v-v_0| \le r, |x-x_0| \le r^3\},
\eeqn
we have

\begin{theo}\label{T41}(Spreading of Positivity)
Let $f(t, x, v)$ be a classical nonnegative solution of
\beqn
\nonumber
\partial_t f - \Delta_v f = -(v + \Phi(x)) \cdot \nabla_x f+ A(t, x, v) \cdot \nabla_v f  + C(t, x, v) f,
\eeqn
in $[0, T) \times \Omega$, where $\Phi(x)$ is Lipschitz
\beqn
\nonumber
|\Phi(x) - \Phi(y)| \le M |x-y|,\quad \forall  x, y \in \R^d.
\eeqn
Suppose further that 
\beqn
\nonumber
A(t, x, v) = \nabla_v W(x, v),
\eeqn
for some $W(x, v)$. Define
\bear
\nonumber
D(t, x, v) =  -\frac 1 4 |A(t, x, v)|^2 -\frac 1 2 \textup{div}_v A(t, x, v ) +\frac 1 2(v+\Phi(x)) \cdot A(t, x, v ) + C(t, x, v).
\eear
For any $(x_0,v_0)$ fixed, define $V=(M+1)^2 (\Phi(0)+|x_0| +|v_0|)$, then for any  $r>0$, $ 0<\tau < \min(1, r^3/2V, log2/M,  1/20M )$, $\alpha>1$, $\delta>0$, there exist $\lambda>0$ only depend on $r^2/\tau, \alpha, M, V$ (independent of $\delta$) such that  if $f \ge \delta > 0$ in $[0, \tau ) \times \bar{B}_r(x_0, v_0)$, then $f \ge K\delta$ in $[\tau /2,  \tau ) \times \bar{B}_{\alpha r}(x_0, v_0)$
where $K$ also depends on $\Vert D \Vert_{L^\infty(\bar{B}_{\lambda r} (x_0, v_0))}$ and $\Vert W \Vert_{L^\infty(\bar{B}_{\lambda r} (x_0, v_0))}$.
\end{theo}

\begin{proof}
See Appendix \ref{secAA}.
\end{proof}

\begin{rem}
Former proofs of spreading of positivity such as Theorem A.19 in \cite{V} assumes that $A$ and $C$ are uniformly bounded, by assuming
\beqn
\nonumber
A(t, x, v) = \nabla_v W(x, v) 
\eeqn
we generalize this theorem to unbounded cases.
\end{rem}

Then we come to prove our main theorem

\begin{theo}\label{T42}
Under the assumption of  Theorem \ref{T12}. The equation (\ref{E13}) defined in Theorem \ref{E12} satisfies the Harris condition: For any $R > 0$, there exist a constant $T = T( R)> 0$ and a positive, nonzero measure $\mu = \mu(R)$ such that
\beqn
\nonumber
S_T f_0 \ge \mu \int_{B_R}
f_0, \quad \forall f_0 \in L^1(m), \quad f \ge 0, \quad \Vert f_0 \Vert_{L^1(m)} \le 1,
\eeqn As a consequence, Theorem \ref{T12} is proved.
\end{theo}
\noindent \begin{proof} By Lemma \ref{L41} we have
there exist $\gamma, \rho > 0 $ such that for any $ t, R > 0$, there exists $(x_0, v_0) \in B_\rho$ such that
\beqn
\nonumber
f(t, x_0, v_0) \ge \gamma \int_{B_R} f_0 . 
\eeqn
where $\gamma, \rho$ does not depend on $f_0, t $ while $x_0, v_0$ may depend on $f_0, t$. By Lemma \ref{L41} we have
\beqn
\nonumber
\Vert f \Vert_{C^{2,\delta}} \le C \Vert f_0 \Vert_{L^1(m)}, \quad \forall t \in [\frac \eta 2,\eta] ,
\eeqn
in particular
\beqn
\nonumber
\Vert \nabla_x f \Vert_{L^\infty} +\Vert \nabla_v f \Vert_{L^\infty} +\Vert \Delta_x f \Vert_{L^\infty} +\Vert f \Vert_{L^\infty} \le C \Vert f_0 \Vert_{L^1(m)}, \quad \forall t \in [\frac \eta 2,\eta] ,
\eeqn
and by  equation
\beqn
\nonumber
\partial_t f :=\LL f  = \partial_x(A(x, v) f ) + \partial_v(B(x, v ) f ) +  \Delta_v f,
\eeqn
we have
\beqn
\nonumber
\Vert \nabla_x f \Vert_{L^\infty(\Omega)} +\Vert \nabla_v f \Vert_{L^\infty(\Omega)} +\Vert \partial_t f \Vert_{L^\infty(\Omega)}  \le C \Vert f_0 \Vert_{L^1(m)} \le C, \quad \forall t \in [\frac \eta 2,\eta],
\eeqn
for $\Omega = B_{2\rho}$ and some constant $C>0$. By continuity, for every $R>0$, there exist $t_1, t_2, r_0,\rho,  \gamma>0$ which do not depend on $f_0$ and $ (x_0, v_0) \in B_\rho$ which may depend on $f_0$, such that for all $t \in (t_1, t_2)$, we have
\beqn
\nonumber
f(t, x, v) \ge \frac {\gamma} 2  1_{B_{r_0}(x_0, v_0 )} \int_{B_R} f_0,
\eeqn
where $B_{r_0 }(x_0, v_0)$ denotes the ball centered at $(x_0, v_0)$ with radius $r_0$. Take $r_1= \min\{(\frac {r_0} 2, (\frac {r_0} 2)^{\frac 1 3} \}$ such that $\bar{B}_{r_1}(x_0, v_0) \subset B_{r_0}(x_0, v_0)$, then we have
\beqn
\nonumber
f(t, x, v) \ge \frac {\gamma} 2  1_{\bar{B}_{r_1}(x_0, v_0 )} \int_{B_R} f_0.
\eeqn
Take $\alpha = \max\{\frac {2\rho} {r_1}, (\frac {2\rho} {r_1})^{\frac 1 3 }, 1 \}$ large such that $B_{2\rho}(x_0, v_0) \subset \bar{B}_{\alpha r_1} (x_0, v_0)$. Define
\beqn
\nonumber
\tau= \min(t_2-t_1, 1, \frac {r_1^3} {2V}, \frac {log2} {M},  \frac {1} {20M} ).
\eeqn
Using Theorem \ref{T41}, we have
\beqn
\nonumber
f(t, x, v) \ge  \frac {\gamma} 2    1_{\bar{B}_{\alpha r_1}(x_0, v_0 )} \int_{B_R} f_0 \ge K \frac {\gamma} 2    1_{B_{2\rho}(x_0, v_0 )} \int_{B_R} f_0,
\eeqn
since $(x_0, v_0) \in B_\rho $ implies that $B_\rho \subset B_{2\rho}(x_0, v_0)$,  we have 
\beqn
\nonumber
f(t, x, v) \ge K \frac {\gamma} 2    1_{B_{\rho}(0, 0)} \int_{B_R} f_0,
\eeqn
for any $ t \in (t_2- \frac  {T} {2} ,t_2)$.
So we can define $\mu(R) =  K \frac {\lambda} 2    1_{B_{\rho}(0, 0)}, T(R)= t_2 - \frac {T} 4$, note it's independent of $(x_0, v_0)$, thus independent of 
$f_0$, we conclude the Harris condition (H2). 
\end{proof}
\noindent Then by Theorem \ref{T21}, we have proved
\beqn
\nonumber
\Vert f(t , \cdot)	- \MM(f_0) G\Vert_{L^1(m)}  \le C e^{- \lambda t } \Vert f_0 -\MM(f_0) G\Vert_{L^1(m)},
\eeqn
which is Theorem \ref{T12} in the case $p=1$.
\begin{rem}
If we replace $f$ in the proofs by the normalized steady state $G$, we can deduce that $G>0$.
\end{rem}

\bigskip
\section{Proof of $L^p(m)$ convergence}
\label{sec5}
\setcounter{equation}{0}
\setcounter{theo}{0}
\bigskip

In the last section, we have prove Theorem \ref{T12} in the case of $p=1$, now we prove it for general $p$, which will complete the proof of the theorem.
In this section,  $A \lesssim B$ will denote $A \le C B$ for some constant $C>0$. First recall the splitting
\beqn
\nonumber
\BB =\LL-\AA, \quad \AA=M_1\chi_{R_1}(x, v),
\eeqn
since
\beqn
\nonumber
\varphi_p(m) \le -C+ M \1_{B_R},
\eeqn
by Remark \ref{R11} it's easily seen that we can take $M_1$, $R_1$ such that
\beqn\label{E51}
\Vert S_\BB(t) \Vert_{L^p(m) \to L^p(m)} \lesssim e^{-at},
\eeqn
and by the Lyapunov condition
\beqn\label{E52}
\Vert S_\BB(t) \Vert_{L^1(m) \to L^1(m)} \lesssim e^{- a t},
\eeqn
for some $\beta>0$. Before going to the proof of our main theorem, we need two last deduced results. 
\begin{lem}\label{L51}
We have
 \beqn
\nonumber
\Vert S_\BB(t)  \AA \Vert_{L^p(m) \to L^p(m)} \lesssim e^{-a t} , \quad \forall t \ge 0,
\eeqn
and
\beqn
\nonumber
\Vert S_\BB(t)  \AA \Vert_{ L^1(m)   \to L^1(m)   } \lesssim e^{-a t}, \quad \forall t \ge 0,
\eeqn
and
\beqn
\nonumber
\Vert S_\BB(t) \AA  \Vert_{ L^1(m)   \to L^p(m)   } \lesssim t^{-\alpha} e^{- a t}, \quad \forall t \ge 0,
\eeqn
for $\alpha =\frac {5d+2} {4}$ and some $\beta > 0$.
\end{lem}

\begin{proof}
\noindent
 The first two inequalities are obtained obviously by (\ref{E51}), (\ref{E52}) and the property of $\AA$. For the third inequality we split it into two parts, $t \in (0, \eta]$ and $t > \eta$, where $\eta$ is defined in Theorem \ref{T31}. When $t \in (0, \eta]$ , we have $e^{-at} \ge e^{-a \eta}$, by Theorem \ref{T31}, we have
\beqn
\nonumber
\Vert  S_\BB(t) \AA  \Vert_{   L^1(m)  \to L^p(m)  }\lesssim t^{-\alpha} \lesssim t^{-\alpha}e^{-a t} ,\quad \forall t \in (0, \eta],
\eeqn
for some $a > 0$. When $ t \ge \eta$,  by Theorem \ref{T31}, we have
\beqn
\nonumber
\Vert  S_\BB(\eta) \Vert_{  L^1(m)   \to L^p(m) }\lesssim  \eta^\alpha \lesssim 1,
\eeqn
and by Lemma \ref{E51}
\beqn
\nonumber
\Vert S_\BB(t - \eta)  \Vert_{L^p(m) \to L^p(m) }\lesssim e^{-a (t-\eta)} \lesssim  e^{-a t},
\eeqn
gathering the two inequalities, we have
\beqn
\nonumber
\Vert S_\BB(t) \AA  \Vert_{L^1(m) \to L^p(m)}\lesssim e^{-at} \lesssim t^{-\alpha}e^{-a t}, \quad \forall t > \eta,
\eeqn
the proof is ended by combining the two cases above.
\end{proof}

\begin{lem}\label{L52}
let $X, Y$ be two Banach spaces, $S(t)$ a semigroup such that for  all $t \ge 0 $ and some $0 < a$ we have
\beqn
\nonumber
\Vert S(t) \Vert_{X \to X} \le C_X e^{-at}, \ \ \Vert S(t) \Vert_{Y \to Y} \le C_Y e^{-at},
\eeqn
and for some $0 < \alpha $, we have 
\beqn
\nonumber
\Vert S(t) \Vert_{X \to Y} \le C_{X,Y} t^{-\alpha} e^{-at}.
\eeqn
Then we can have that for all integer $ n  > 0$
\beqn
\nonumber
\Vert S^{(*n)}(t) \Vert_{X \to X} \le C_{X,n} t^{n-1} e^{-at},
\eeqn
similarly
\beqn
\nonumber
\Vert S^{(*n)}(t) \Vert_{Y \to Y} \le C_{Y,n} t^{n-1} e^{-at},
\eeqn
and
\beqn
\nonumber
\Vert S^{(*n)}(t) \Vert_{X \to Y} \le C_{X,Y,n} t^{n-\alpha-1} e^{-at}.
\eeqn
In particular for $\alpha+1 < n$, and for any $ a^{*} < a$
\beqn
\nonumber
\Vert S^{(*n)}(t) \Vert_{X \to Y} \le C_{X,Y,n} e^{-a^* t}.
\eeqn
\end{lem}

\begin{proof}
See Lemma 2.5 in \cite{MQT}.
\end{proof}
Then we come to the final proof.

\begin{proof}
\noindent ({\bf Proof of Theorem \ref{T12}.})
Remember that we already proved
\beqn
\nonumber
\Vert S_\LL(I - \Pi)(t) \Vert_{L^1(m)\to L^1(m)} \lesssim  e^{-at},
\eeqn
where $I$ is the identity operator and $\Pi$ is a projection operator defined by
\beqn
\nonumber
\Pi (f) =\MM(f) G.
\eeqn 
First, Iterating the Duhamel's formula we split it into 3 terms
\bear
\nonumber
S_\LL(I-\Pi) &=&(I-\Pi)\{S_\BB + \sum_{l=1}^{n-1}( S_\BB \AA)^{(*l)}* (S_\BB) \}
\\ \nonumber
&&+( S_\BB(t) \AA )^{(*(n-1))} * \{ (I-\Pi)S_\LL \}*(\AA S_\BB(t)),
\eear
and we will estimate them separately. By (\ref{E51}) the first term is thus estimated. For the second term, still using (\ref{E51}), we get
\beqn
\nonumber
\Vert S_\BB(t)\AA \Vert_{L^p(m) \to L^p(m)  } \lesssim e^{-at},
\eeqn
by Lemma \ref{L52}, we have
\beqn
\nonumber
\Vert (S_\BB(t)\AA)^{(*l)} \Vert_{L^p(m) \to L^p(m ) } \lesssim t^{l-1} e^{-at },
\eeqn
together with (\ref{E51}) the second term is estimated. For the last term by H\"older's inequality we have
\beqn
\nonumber
\Vert I \Vert_{L^p(m) \to L^1(mG)} \lesssim  1
\eeqn
with
\beqn
\nonumber
G=e^{-(|v|^2+|x|^2)}
\eeqn
(there are many choice of $G$) so we have
\beqn
\nonumber
\Vert \AA S_\BB(t) \Vert_{L^p( m) \to L^1(m)} \lesssim e^{- a t} .
\eeqn
By Lemma \ref{L51} and \ref{L52}, we have
\beqn
\nonumber
\Vert (S_\BB \AA )^{(*(n-1))}(t) \Vert_{L^1(m) \to L^p(m)   } \lesssim t^{n-\alpha-2}e^{- a t} ,
\eeqn
finally recall 
\beqn
\nonumber
\Vert S_\LL(t)(I-\Pi) \Vert_{  L^1(m)   \to L^1(m) } \lesssim e^{-at}.
\eeqn
Taking $n > \alpha+2 $  the third term is estimated, thus the proof is ended by gathering the inequalities above. 
\end{proof}

\bigskip
\section{Proof of Main Theorem}
\label{sec6}
\setcounter{equation}{0}
\setcounter{theo}{0}

\bigskip

This section we come to prove Theorem \ref{T11}. Recall that we have proved Theorem \ref{T12} in the last section, the only thing remain to prove is to find a weight function  $m$ and a function $H \ge 1$ in Theorem \ref{T12}. 

\begin{theo}\label{T61}
Denote $\LL$ the operator of the kinetic Fitzhugh-Nagumo equation (\ref{E12}), then there exist a weight function $m$ and a function $H\ge 1$ satisfies Theorem \ref{T12}.
\end{theo}

\begin{proof}
We recall the kinetic Fitzhugh-Nagumo equation 
\beqn
\nonumber
\partial_t f :=\LL f  = \partial_x(A(x, v) f ) + \partial_v(B(x, v ) f ) +  \partial^2_{vv} f,
\eeqn
with 
\beqn
\nonumber
A(x, v) = ax- bv, \quad B(x, v) =v(v-1) (v-c) +x,
\eeqn
by a change of variable $w=bv$, the equation is changed to
\beqn
\nonumber
\partial_t f :=\LL f  = \partial_x(A(x, v) f ) + \partial_v(B(x, v ) f ) + \frac 1 {b^2} \partial^2_{vv} f,
\eeqn
with 
\beqn
\nonumber
A(x, v) = ax- v, \quad B(x, v) =\frac 1 {b^3} v(v-b) (v-bc) +x,
\eeqn
we have
\beqn
\nonumber
\LL^*f = -A(x,v) \partial_x f - B(x,v )\partial_v f +\frac 1 {b^2} \partial_{vv} f,
\eeqn
for some $a, b , c >0$. This time we have
\bear
\nonumber
\frac {\LL^*m} {m} &=& v \cdot \frac {\nabla_x m} m - a x \cdot \frac {\nabla_x m} {m}   + \frac 1 {b^2} \frac{\Delta_v m} {m} 
\\ \nonumber
 &&- (\frac 1 {b^3} v(v-b) (v-bc) +x) \cdot \frac{\nabla_v m } {m}.
\eear
We can take $m=e^{\frac r 2(|x|^2 +|v|^2)}$, with $r>0$ to be fixed later, then we have
\beqn
\nonumber
\frac {\nabla_x m } {m} = r x, \quad \frac {\nabla_v m } {m} = r v, \quad \frac {\Delta_v m } {m} = r + r^2|v|^2,
\eeqn
we then compute
\bear
\nonumber
\frac {\LL^*m} {m} &=& r x \cdot v - ar |x|^2 + \frac r {b^2} + \frac {r^2} {b^2}|v|^2 
\\ \nonumber
&-&\frac 1 {b^3} |v|^2(v-b) (v-bc) - r  x \cdot v 
\\ \nonumber
&=& -ar|x|^2 - \frac 1 {b^3} |v|^4 + M_1 v^3 +M_2|v|^2 + M_3,
\eear
for some constant $M_1, M_2, M_3>0$, so the Lyapunov condition (C1)
\beqn
\nonumber
\LL^*m \le -\alpha m +b,
\eeqn
is satisfied for some $\alpha, b>0$, similarly
\bear
\nonumber
\phi_2(m)  &=& v \cdot \frac {\nabla_x m} m - a x \cdot \frac {\nabla_x m} {m} + \frac a 2  +\frac 1 {b^2} \frac {|\nabla_v m|^2} {m^2}+ \frac 1 {b^2} \frac{\Delta_v m} {m} 
\\ \nonumber
 &&- (\frac 1 {b^3} v(v-b) (v-bc) +x) \cdot \frac{\nabla_v m } {m} + \frac 1 {2b^3} (3v^2 + 2b(1+c) v +b^2c),
\eear
and this time we have
\bear
\nonumber
\phi_2(m) &=& r x \cdot v - ar |x|^2 + \frac a 2 + \frac {r^2} {b^2}  |v|^2+ \frac r {b^2} + \frac {r^2} {b^2}|v|^2 
\\ \nonumber
&-&\frac 1 {b^3} |v|^2(v-b) (v-bc) - r  x \cdot v + \frac 1 {2b^3} (3v^2 + 2b(1+c) v +b^2c)
\\ \nonumber
&=& -ar |x|^2 - \frac 1 {b^3} |v|^4 + K_1 v^3 +K_2|v|^2 + K_3v +K_4,
\eear
for some constants $K_1, K_2, K_3, K_4$, if we take 
\beqn
\nonumber
H = |v|^4 +|x|^2+1,
\eeqn
it's easily seen that we have
\beqn
\nonumber
-C_1 H \le \phi_2(m)  \le -C_2 H  +C_3,
\eeqn
for some $C_1, C_2, C_3 >0$, which is just condition (C2). And it's easily seen that for any integer $n \ge 2$ fixed, for any $\epsilon>0$ small, we can find a constant $C_{\epsilon, n} $ such that
\beqn
\nonumber
\sum_{k=1}^n |D_x^k (ax )| +\sum_{k=1}^n |D_{x, v}^k (\frac 1 {b^3} v(v-b) (v-bc) +x)| \le P_1|v|^2 + P_2 |v| + P_3 \le C_{n,\epsilon} + \epsilon H,
\eeqn
with $P_1, P_2, P_3>0$ constant, so condition (C3) is also satisfied. Since
\beqn
\nonumber
\frac {\Delta_{x, v} m} {m }  =  2r + r^2|v|^2 +r^2|x|^2\ge 0,
\eeqn
all the conditions of Theorem \ref{T12} is satisfied, we finally compute
\bear
\nonumber
\varphi_{\infty}(m)  &=& v \cdot \frac {\nabla_x m} m - ax \cdot \frac {\nabla_x m} {m} +  a +  \frac 2 {b^2} \frac {|\nabla_v m|^2} {m^2}- \frac 1 {b^2}\frac{\Delta_v m} {m} 
\\ \nonumber
&&- (\frac 1 {b^3} v(v-b) (v-bc) +x)  \cdot \frac{\nabla_v m } {m} + \frac 1 {b^3} (3v^2 + 2b(1+c) v +b^2c).
\eear
We have
\bear
\nonumber
\varphi_{\infty}(m) &=& r x \cdot v - ar |x|^2 + a + \frac {2r^2} {b^2}  |v|^2 - \frac r {b^2} - \frac {r^2} {b^2}|v|^2 
\\ \nonumber
&-&\frac 1 {b^3} |v|^2(v-b) (v-bc) - r  x \cdot v + \frac 1 {b^3} (3v^2 + 2b(1+c) v +b^2c)
\\ \nonumber
&=& -ar |x|^2 - \frac 1 {b^3} |v|^4 + K_1 v^3 +K_2|v|^2 + K_3v +K_4,
\eear
for some constants $K_1, K_2, K_3, K_4$, it's easily seen that
\beqn
\nonumber
\varphi_{\infty}(m) \le C+M\1_{B_R},
\eeqn
for some $C, M, R>0$, the proof is finished.
\end{proof}

Then we come to find a weight function  $m$ and a function $H \ge 1$ for the kinetic Fokker-Planck equation with general force.

\begin{theo}\label{T62}
Denote $\LL$ the operator of the kinetic Fokker-Planck equation (\ref{E11}), then there exist a weight  function $m$ and a function $H$  satisfies Theorem \ref{T12}.
\end{theo}
\begin{proof}
First we have
\beqn
\nonumber
\LL^* f = v \cdot \nabla_x f - \nabla_x V(x) \cdot \nabla_v f +\Delta_v f - \nabla_v W(v) \cdot \nabla_v f,
\eeqn
 denote
\beqn
\nonumber
H_1=\frac {|v|^2}{2} + V(x)+ \epsilon v \cdot \nabla_x  \langle x \rangle, \quad m=e^{\lambda H_1},
\eeqn
so we have
\beqn
\nonumber
\frac {\LL^* m} {m} =  \lambda (v \cdot \nabla_x H_1 - \nabla_x V(x) \cdot \nabla_ v H_1+ \Delta_v H_1 + \lambda |\nabla_v H_1|^2 - \nabla_v W(v) \cdot \nabla_v H_1).
\eeqn
We easily compute
\beqn
\nonumber
\nabla_v H_1 = v+ \epsilon \nabla_x  \langle x \rangle, \quad \nabla_x H_1 = \nabla_x V(x) +\epsilon v \cdot \nabla_x^2   \langle x \rangle, \quad \Delta_v H_1=d,
\eeqn
and since
\beqn
\nonumber
\nabla_x^2 \langle x \rangle \le CI,
\eeqn
where $I$ is the $d \times d$ identity matrix, we have
\bear
\nonumber
\frac {\LL^* m} {m} &=&  \lambda( v \cdot \nabla_x V(x) +\epsilon v \nabla_x \langle x \rangle^2 v- \nabla_x V(x) \cdot v -\nabla_x V(x) \cdot \nabla_x \langle x \rangle+d
\\ \nonumber
&&+ \lambda |v+ \epsilon \nabla_x  \langle x \rangle|^2 - \epsilon \nabla_v W(v) \cdot  \nabla_x \langle x \rangle - \nabla_v W(v) \cdot v
\\ \nonumber
&&\le C(\lambda^2|v|^2 +\lambda |\nabla W(v)| )-\lambda \nabla_x V(x) \cdot \nabla_x \langle x \rangle -\lambda \nabla_v W(v) \cdot v,
\eear
for some constant take $\lambda>0$ small, we conclude
\beqn
\nonumber
\LL^* m \le -C_1 H m+C_2,
\eeqn
for some constant $C_1, C_2>0$, with $H =  \langle v \rangle^\beta + \langle x  \rangle^{\gamma-1}+1$, then the Lyapunov condition (C1) follows. For the second inequality, by Lemma \ref{L31} we have
\bear
\nonumber
\phi_2 (m)&=&  \lambda (v \cdot \nabla_x H_1+ \nabla_x V(x) \cdot \nabla_ v H_1 + \Delta_v H_1 
\\ \nonumber
&&+ 2\lambda |\nabla_v H_1|^2 - \nabla_v W(v) \cdot \nabla_v H_1)+\frac 1 2\Delta_v W(v),
\eear
we compute
\bear
\nonumber
\phi_2 (m)&=&  \lambda( v \cdot \nabla_x V(x) +\epsilon v  \nabla_x \langle x \rangle v- \nabla_x V(x) \cdot v -\nabla_x V(x) \cdot \nabla_x \langle x \rangle+d
\\ \nonumber
&& +2 \lambda |v+ \epsilon \nabla_x  \langle x \rangle|^2 - \epsilon \nabla_v W(v) \cdot  \nabla_x \langle x \rangle - \nabla_v W(v) \cdot v) + \frac 1 2\Delta_v W(v)
\\ \nonumber
&\le& C(\lambda^2|v|^2 +\lambda |\nabla W(v)|+ |\Delta_x W(v)| )-\lambda \nabla_x V(x) \cdot \nabla_x \langle x \rangle -\lambda \nabla_v W(v) \cdot v,
\eear
and we still have
\beqn
\nonumber
-C_1 H \le \phi_2(m) \le -C_2 H  +C_3,
\eeqn
for some constant $C_1, C_2, C_3>0$, thus condition(C2) is proved. It's easily seen that for any integer $n \ge 2$ fixed, for any $\epsilon>0$ small, we can find a constant $C_{\epsilon, n} $ such that
\beqn
\nonumber
\sum_{k=1}^n |D_x^k \nabla_x V(x) | +\sum_{k=1}^n |D_{x, v}^k \nabla_v W(v)|  = \sum_{k=2}^{n+1} |D_x^k V(x) | +\sum_{k=2}^{n+1} |D_{x, v}^k W(v)| \le C_{n,\epsilon} + \epsilon H,
\eeqn
by the definition of $V(x)$ and $W(v)$, so condition (C3) is also satisfied. For the last condition
\beqn
\nonumber
\frac {\Delta_{x, v} m} {m }  = \lambda^2|\nabla_x H_1|^2+\lambda^2|\nabla_v H_1|^2 +\lambda \Delta_x H_1  + \lambda \Delta_v H_1.
\eeqn
For the term $\Delta_x H_1$ we compute
\beqn
\nonumber
\Delta_x H_1= \nabla_x V(x) + \epsilon  v \cdot  \nabla \Delta \langle x \rangle \ge -L_1 - L_2|v|,
\eeqn
for some constant $L_1, L_2>0$, and 
\beqn
\nonumber
|\nabla_v H_1|^2= |v+ \epsilon \nabla_x  \langle x \rangle|^2 \ge \frac {|v|^2} 2-L_3,
\eeqn
for some constant $L_3>0$, since 
\beqn
\nonumber
\Delta_v H_1 =d \ge 0, 
\eeqn
we conclude that
\beqn
\nonumber
\frac {\Delta_{x, v} m} {m }  \ge C,
\eeqn
for some constant $C$, so all the conditions are satisfied. We finally compute
\bear
\nonumber
\varphi_{\infty}(m)  &=& \lambda( v \cdot \nabla_x V(x) +\epsilon v  \nabla_x \langle x \rangle v- \nabla_x V(x) \cdot v -\nabla_x V(x) \cdot \nabla_x \langle x \rangle - d
\\ \nonumber
&& + \lambda |v+ \epsilon \nabla_x  \langle x \rangle|^2 - \epsilon \nabla_v W(v) \cdot  \nabla_x \langle x \rangle - \nabla_v W(v) \cdot v) +\Delta_v W(v)
\\ \nonumber
&\le& C(\lambda^2|v|^2 +\lambda |\nabla W(v)|+ |\Delta_x W(v)| )-\lambda \nabla_x V(x) \cdot \nabla_x \langle x \rangle -\lambda \nabla_v W(v) \cdot v ,
\eear
the proof is thus finished.
\end{proof}

    \begin{appendix}

\bigskip
\section{Proof of spreading of positivity}
\label{secAA}
\setcounter{equation}{0}
\setcounter{theo}{0}

In this section, we will use the notation
\beqn
\nonumber
\bar{B}_r(x_0, v_0) = \{ (x, v) \in \R^d \times \R^d : |v-v_0| \le r, |x-x_0| \le r^3\},
\eeqn
and $\bar{B}_r$ will stand for $\bar{B}_r(x_0, v_0)$. 
Before proving the theorem on spreading of positivity, we first prove a useful lemma.

\begin{lem}\label{LA1}
Define $X_t(x_0, v_0)$ (abbreviated $X_t$ in the sequel) in this way, consider the ordinary differential equation
\beqn
\nonumber
\frac {d x} {dt} = v_0 + \Phi(x),
\eeqn
and denote by $X_t(x_0, v_0)$ the solution to this ordinary differential equation at time $t$ with $x(0) = x_0$, where $\Phi(x)$ is Lipschitz
\beqn
\nonumber
|\Phi(x) - \Phi(y)| \le M |x-y|,\quad \forall  x, y \in \R^d.
\eeqn
with loss of generality we assume $M \ge 1$. Then we have, for any $(x_0, v_0) \in \R^d$ fixed, $t \in[0, \min \{\frac {log 2} {M}, 1\})$, we have
\beqn
\nonumber
|X_t -x_0| \le t  (M+1)^2 (|v_0|+|x_0| +|\Phi(0)|),
\eeqn
\end{lem}
\begin{proof} 
Since $\Phi(x)$ is Lipschitz, the existence and uniqueness of $X_t$ is satisfied. First by the definition of $X_t$ we have
\beqn
\nonumber
\frac {d|X_t|} {dt} \le | \frac {dX_t} {dt}| \le |v_0| +M(|X_t|+|\Phi(0)|),
\eeqn
by Gr\"onwall's lemma we have
\beqn
\nonumber
|X_t| \le e^{Mt}|x_0| +\frac {1} M (e^{Mt} - 1)(|v_0| +M|\Phi(0)|) ,
\eeqn
since $M\ge1$, so for $t \in(0, \frac {log 2} M)$ we have
\beqn
\nonumber
|X_t| \le 2|x_0| +|v_0| +M|\Phi(0)|,
\eeqn
so
\bear
\nonumber
 | \frac {dX_t} {dt}| \le |v_0| +M(|X_t|+|\Phi(0)|) &\le& (M+1) |v_0| +2M|x_0| +M^2|\Phi(0)| 
\\ \nonumber
&\le& (M+1)^2 (|v_0|+|x_0| +|\Phi(0)|),
\eear
for any $t \in(0, \frac {log 2} {M} )$, the lemma is thus proved.
\end{proof}
Use this $X_t$, we come to construct a subsolution which is useful in our proof.

\begin{lem}\label{LA2}
Define operator $\LL$ as
 \beqn
 \nonumber
\LL =\frac{ \partial} {\partial t} + (v + \Phi(x)) \cdot \nabla_x  - \Delta_v,
 \eeqn
where $\Phi(x)$ is Lipschitz
\beqn
\nonumber
|\Phi(x) - \Phi(y)| \le M |x-y|,\quad \forall  x, y \in \R^d.
\eeqn
Then for any $(x_0, v_0) \in \R^d$ fixed, define $V=(M+1)^2 (\Phi(0)+|x_0| +|v_0|)$, then for any $r>0$, $ 0<\tau < \min(1, r^3/2V, log2/M,  1/20M )$,$\alpha>1$, $\delta>0$, there exist constants $\lambda>\alpha, K>0$ which only depend on $r^2/\tau$, $V$, $M$, $\alpha$ (independent of $\delta$) and a function $\phi$ such that
\beqn
\nonumber
\LL \phi \le 0, \quad  \textup{in} \  [0, \tau) \times (\bar{B}_{\lambda r} \setminus \bar{B}_{r}),
\eeqn
and some boundary conditions
\beqn
\nonumber
\phi \le 0, \quad \textup{on} \ t=0, \quad \phi \le \delta \quad  \textup{on} \ [0, \tau ) \times \partial \bar{B}_{r}, \quad \phi \le 0 \quad \textup{on} \  [0, \tau) \times \partial \bar{B}_{\lambda r},
\eeqn
while
\beqn
\nonumber
\phi \ge K\delta \quad \textup{on} \ [\frac \tau 2, \tau) \times (\bar{B}_{\alpha r} \setminus \bar{B}_{r}).
\eeqn

\end{lem}

\begin{proof} This proof is similar to the proof in \cite{V} Appendix A. 22.
For $t \in  (0, \tau]$ and $(x, v) \in \R^d \setminus \bar{B}_r$ let
\beqn
 \nonumber
 Q(t, x, v) = a \frac {|v - v_0|^2} {2t} - b \frac {( v- v_0, x-X_t(x_0, v_0) )} {t^2}  +c \frac {|x - X_t(x_0, v_0)|^2} {2t^3},
\eeqn
where $a, b, c > 0$ will
be chosen later on, and we define $X_t(x_0, v_0)$ (abbreviated $X_t$ in the sequel) in this way, consider the ordinary differential equation
\beqn
\nonumber
\frac {d x} {dt} = v_0 + \Phi(x),
\eeqn
and denote by $X_t(x_0, v_0)$ the solution to this ordinary differential equation at time $t$ with $x(0) = x_0$.  Let further
\beqn
\nonumber
\phi(t, x, v) = \delta e^{- \mu Q(t, x, v)} - \epsilon,
\eeqn
where $\mu, \epsilon > 0$ will be chosen later on. Let us assume $b^2 < ac$, so that
$Q$ is a positive definite quadratic form in the two variables $v - v_0$ and $x - X_t$. Then
\beqn
\nonumber
\LL \phi = -\mu \delta e^{- \mu Q} \AA(Q),
\eeqn
where
\beqn
\nonumber
\AA(Q) = \partial_t Q + (v + \Phi(x)) \cdot \nabla_x Q -\Delta_v Q +\mu |\nabla_v Q|^2.
\eeqn
By computation,
\bear
\nonumber
\AA(Q) &=& -a \frac {|v-v_0|^2} {2t^2} + 2b \frac {( v-v_0, x-X_t )} {t^3}  -3c \frac {|x-X_t|^2} {2t^4}
\\ \nonumber
&+& b \frac { ( v-v_0, v_0 +\Phi(X_t) )} {t^2} - c \frac {( x-X_t, v_0 +\Phi(X_t) )} {t^3}
\\ \nonumber
&-& b \frac { ( v-v_0, v +\Phi(x) )} {t^2} + c \frac {( x-X_t, v +\Phi(x) )} {t^3} - a \frac {d} {t}
\\ \nonumber
&+&  \mu |a\frac {v-v_0} {t} - b\frac {x-X_t } {t^2}|^2 
\\ \nonumber
&=& \BB ( \frac {v-v_0} {t}, \frac {x-X_t} {t^2} )+ c\frac { ( x-X_t, \Phi(x)  - \Phi(X_t) )} {2t^3}
\\ \nonumber
&-&b \frac { ( v-v_0, \Phi(x)  - \Phi(X_t) )} {t^2}  -a\frac d t,
\eear
where $\BB$ is a quadratic form on $\R^n \times \R^n$ with matrix $P \otimes I_n$,

\[ 
P=\begin{pmatrix}
{\mu a^2-\frac a 2 -b}& {-\mu ab +b +\frac c 2} \\{-\mu ab +b +\frac c 2} &{\mu b^2 - \frac {3c} 2}
\end{pmatrix}
\]
If a, b, c are given, then as  $\mu \to  \infty$
\beqn
\nonumber
\left\{ \begin{aligned} &\hbox{tr} P = \mu (a^2 + b^2) + O(1),\\
&\hbox{det}P  = \mu(\frac{3ab^2} 2 + abc -b^3 - \frac {3a^2c} 2 )+ O(1).
\end{aligned} \right.
\eeqn
Both quantities are positive if $b \ge 2a$ and $ac \gg b^2$, for example we can take $b=2a,  c > 12b$, then as $\mu \to \infty$
\beqn
\nonumber
\left\{ \begin{aligned} &\hbox{tr} P = 5\mu a^2 + O(1),\\
&\hbox{det}P  =  \mu (\frac 1 2 a^2c-a^2 b) + O(1) \ge \mu \frac 1 3 a^2c+ O(1).
\end{aligned} \right.
\eeqn
the eigenvalues of $M$ are of order $5 \mu a^2$ and $\frac c {15}$. So we
may choose a, b, c and $\mu$ so that
\beqn
\nonumber
\BB(\frac {v-v_0} {t}, \frac {x-X_t} {t^2}) \ge \frac {c} {20} (\frac {
|v - v_0|^2} {t^2} + \frac {|x - X_t|^2} {t^4}).
\eeqn
If $\tau \le \frac 1 {20M}$, we have
\beqn
\nonumber
c\frac { |( x-X_t, \Phi(x)  - \Phi(X_t) )|} {2t^3} \le \frac {Mt} {2} c\frac {|x - X_t|^2} {t^4} \le \frac c {40} \frac {|x - X_t|^2} {t^4},
\eeqn
gathering the two terms we have
\bear
\nonumber
\BB(\frac {v-v_0} {t}, \frac {x-X_t} {t^2})+ c\frac { ( x-X_t, \Phi(x)  - \Phi(X_t) )} {2t^3}  
&\ge& \frac c {40} (\frac {
|v - v_0|^2} {t^2} + \frac {|x - X_t|^2} {t^4}).
\eear
If $\tau \le 1$
\beqn
\nonumber
-b \frac { ( v-v_0, \Phi(x)  - \Phi(X_t) )} {t^2}  -a\frac d t\ge -bM^2\frac {|x - X_t|^2} {t^4} - b\frac {
|v - v_0|^2} {t^2}  -\frac {b d } {2t},
\eeqn
gathering the two terms, suppose $c \ge 80(M+1)^2 b$, we have
\beqn
\nonumber
\AA(Q)≥ \ge  \frac {bd} {2t} [ \frac {c} {40bd} (\frac {
|v - v_0|^2} {t} + \frac {|x - X_t|^2} {t^3})- 1].
\eeqn
Recall that $(x, v) \notin \bar{B}_r$, so\\
- either $|v- v_0| \ge r$, and then $\AA(Q) \ge \frac {bd} {2t}[\frac {c} {40bd } r^2/\tau - 1]$, which
is positive if $ c \ge 40bd \frac {\tau}{r^2}$;\\
- or $|x - x_0| \ge r^3$, and then by Lemma \ref{LA1}, for any $\tau \le  \min \{1, r^3/(2V ), log2/M \}$
\bear
\nonumber
\frac {|x - X_t|^2}{t^2} &\ge&  \frac {|x - x_0|^2} {t^2} -\frac {|X_t - x_0|^2} {t^2} 
\\ \nonumber
&\ge& \frac {|x - x_0|^2} {t^2}- ((M+1)^2 (|v_0|+|x_0| +|\Phi(0)|))^2 \ge \frac {r^6} {\tau^2} - V^2 \ge \frac {r^6}  {2\tau^2},
\eear
so $\AA(Q) \ge \frac {bd} {2t} [\frac {c}{40 b d} \frac {r^6}{2\tau^3} - 1]$, which is positive as soon as $c \ge
80 b d(\frac {\tau} {r^2})^3$.
so it's OK to take 
\beqn
\nonumber
c=  b \max \{12, 80(M+1)^2, 80d (\frac {\tau } {r^2} )^3, 40d \frac \tau {r^2} \}.
\eeqn
To summarize: under our assumptions there is a way to choose the
constants $a, b, c, \mu$, depending only on $d, M, V, r^2/\tau$ , satisfying $c > b > a $ and $ac > b^2$, so that 
\beqn
\nonumber
\LL \phi \le 0, \quad  \hbox{in} \  [0, \tau) \times (\bar{B}_{\lambda r} \setminus \bar{B}_{r}),
\eeqn
as soon as $ 0<\tau < \min(1, r^3/2V, log2/M,  1/20M )$.  Recall that
\beqn
\nonumber
\phi(t, x, v) = \delta e^{- \mu Q(t, x, v)} - \epsilon.
\eeqn
The boundary condition at $t = 0$ is obvious since $e^{- \mu Q(t, x, v)} $ vanishes identically at $t=0$
(more rigorously, $e^{- \mu Q(t, x, v)} $ can be extended by continuity by 0 at
$t = 0$).  The condition is also true on $[0, \tau) \times \partial \bar{B}_r$ since $\phi \le \delta$. It
remains to prove it on $[0, \tau) \times \partial \bar{B}_{\lambda r}$ . For that we estimate $Q$ from below, since $c>12b, b=2a$, it's easily to seen that for any $(t, x, v) \in [0,\tau) \times \partial \bar{B}_{\lambda r}$
\beqn
\nonumber
Q(t, x, v) \ge \frac {a} 4(\frac { |v - v_0|^2} {t} + \frac {|x-X_t|^2} {t^3}) \ge \frac {a} 4 \min (\frac {\lambda^2r^2} {\tau}, \frac {\lambda^6r^6} {2\tau^3} ) \ge \frac{a  \lambda^2} {8} \min(\frac {r^2} {\tau}, \frac {r^6} {\tau^3}).
\eeqn
Thus if we choose
\beqn
\nonumber
\epsilon = \delta \exp (- \frac {\mu a \lambda^2} {8} \min(\frac {r^2}{\tau}, \frac {r^6} {\tau^3})),
\eeqn
we make sure that $\phi = \delta e^{-\mu Q} -\epsilon \le 0$ on $[0,\tau) \times \partial \bar{B}_{\lambda r}$.
We finally come to prove the last thing, indeed, if $t \ge \tau/2$ and $(x, v) \in \bar{B}_{\alpha r} \setminus \bar{B}_{r}$ then
\beqn
\nonumber
Q(t, x, v) \le 2c(\frac {|v-v_0|^2} {t} + \frac {|x-X_t|^2} {t^3})
\le 2 c( \frac {2 \alpha^2 r^2} {\tau} + \frac {9\alpha^6 r^6} {\tau^3}) \le22 \alpha^6 c \max ( {\frac {r^2} \tau, \frac {r^6}   {\tau^3}}),
\eeqn
take $\lambda$ such that
\beqn
\nonumber
22 \times 8 \times 2 \alpha^6 c \max(\frac {r^2} {\tau}, \frac {r^6} {\tau^3} ) = a \lambda^2 \min (\frac{r^2} {\tau}, \frac  {r^6} {\tau^3}),
\eeqn
so we can find $K>0$ such that
\beqn
\nonumber
\phi  \ge \delta [\exp(-22 \alpha^6 a \mu c \max (\frac{r^2} {\tau}, \frac {r^6} {\tau^3})) - \exp( - \frac {\mu a \lambda^2} {8} \min (\frac {r^2} {\tau}, \frac {r^6} {\tau^3}))] \ge K \delta,
\eeqn
recall the relationship between $a$ and $c$, we conclude that $\lambda, K $ depends only on $r^2/\tau, M ,V, \alpha$, the proof is thus ended.
\end{proof}

\begin{theo}\label{TA1}
Let $f(t, x, v)$ be a classical nonnegative solution of
\beqn
\nonumber
\partial_t f - \Delta_v f = -(v + \Phi(x)) \cdot \nabla_x f+ A(t, x, v) \cdot \nabla_v f  + C(t, x, v) f,
\eeqn
in $[0, T) \times \Omega$, where $\Phi(x)$ is Lipschitz
\beqn
\nonumber
|\Phi(x) - \Phi(y)| \le M |x-y|,\quad \forall  x, y \in \R^d, 
\eeqn
suppose further that 
\beqn
\nonumber
A(t, x, v) = \nabla_v W(x, v),
\eeqn
for some $W(x, v)$. Define
\bear
\nonumber
D(t, x, v) =  -\frac 1 4 |\nabla_v A(t, x, v)|^2 -\frac 1 2 \textup{div}_v A(t, x, v ) +\frac 1 2(v+\Phi(x)) \cdot A(t, x, v ) + C(t, x, v),
\eear
Then for any $(x_0,v_0)$ fixed, $r>0$, $ 0<\tau < \min(1, r^3/2V, log2/M,  1/20M )$, $\alpha>1$, $\delta>0$, there exist $\lambda>0$ only depend on $r^2/\tau, \alpha, M, V$ (independent of $\delta$) such that  if $f \ge \delta > 0$ in $[0, \tau ) \times \bar{B}_r(x_0, v_0)$, then $f \ge K\delta$ in $[\tau /2,  \tau ) \times \bar{B}_{\alpha r}(x_0, v_0)$
where $K$ also depends on $\Vert D \Vert_{L^\infty(\bar{B}_{\lambda r} (x_0, v_0))}$ and $\Vert W \Vert_{L^\infty(\bar{B}_{\lambda r} (x_0, v_0))}$.
\end{theo}

\begin{proof}
We first start by a taking $f =  hG$, with $G=e^{-\frac 1 2 W(x, v) }$, we have
\beqn
\nonumber
\nabla_x G = -\frac 1 2 \nabla_x W(x, v ) G, \quad  \nabla_v G = -\frac 1 2 \nabla_v W(x, v) G,
\eeqn
and
\beqn
\nonumber
 \Delta_v G = \frac 1 4|\nabla_v W(x, v)|^2-\frac 1 2 \Delta_v W(x, v),
\eeqn
the equation turns to
\bear
\nonumber
G\partial_t h &=& G \Delta_v h+ 2 \nabla_v h \cdot \nabla_v G +h\Delta_v G - G (v + \Phi(x)) \cdot \nabla_x h - h (v + \Phi(x)) \cdot \nabla_x G
\\ \nonumber
&&+G A(t, x, v) \cdot \nabla_v h +  \nabla_v G \cdot A(t, x, v) h +C(t, x, v )G h.
\eear
By the definition of $G$ we have
\beqn
\nonumber
2\nabla_v G \cdot \nabla_v h =- G A(t, x, v) \cdot \nabla_v h,
\eeqn
so the equation will turns to
\beqn
\nonumber
\partial_t h = \Delta_v h  -(v + \Phi(x)) \cdot \nabla_x h + D(t, x, v) h,
\eeqn
with 
\bear
\nonumber
D(t, x, v) &=& \frac 1 4|\nabla_v W(x, v)|^2-\frac 1 2 \Delta_v W(x, v)+ \frac 1 2(v+\Phi(x)) \cdot \nabla_v W(x, v)
\\ \nonumber
&&- \frac 1 2\nabla_v W(x, v) \cdot A(t, x, v) +C(t, x, v)
\\ \nonumber
&=&  -\frac 1 4 |A(t, x, v)|^2 -\frac 1 2 \textup{div}_v A(t, x, v ) +\frac 1 2(v+\Phi(x)) \cdot A(t, x, v ) 
\\ \nonumber
&&+ C(t, x, v).
\eear
Take the $\lambda$ from Lemma \ref{LA2}. Then take $\bar{D} = \Vert D \Vert_{L^\infty(\bar{B}_{\lambda r}(x_0, v_0 ))}$, and $\bar{E}= e^{\Vert W \Vert_{L^\infty(\bar{B}_{\lambda r} (x_0, v_0))}}$, then we have
\beqn
\nonumber
h \ge \frac \delta {\bar{E}}, \quad \textup{in} \ [0,\tau ) \times \bar{B}_r.
\eeqn
Let $g(t, x, v) = e^{\bar{D}t} h(t, x, v)$, then $g \ge h$ and $\LL g \ge 0$ in $(0, \tau) \times \bar{B}_{\lambda r}(x_0, v_0 )$, where
 \beqn
 \nonumber
\LL =\frac{ \partial} {\partial t} + (v + \Phi(x)) \cdot \nabla_x  - \Delta_v,
 \eeqn
by Lemma \ref{LA2}, we can find a $\phi$ such that
\beqn
\nonumber
\LL \phi \le 0, \quad  \hbox{in} \  [0, \tau) \times (\bar{B}_{\lambda r} \setminus \bar{B}_{r}),
\eeqn
and
\beqn
\nonumber
\phi \le 0, \quad \textup{on} \ t=0, \quad \phi \le \frac {\delta} {\bar{E}} \quad  \textup{on} \ [0, \tau ) \times \partial \bar{B}_{r}, \quad \phi \le 0 \quad \textup{on} \  [0, \tau) \times \partial \bar{B}_{\lambda r},
\eeqn
while
\beqn
\nonumber
\phi \ge K\frac {\delta} {\bar{E}} \quad \textup{on} \quad  [\frac \tau 2, \tau) \times (\bar{B}_{\alpha r} \setminus \bar{B}_{r}).
\eeqn
So $\phi$ is a subsolution to $g$, then we have
\beqn
\nonumber
g \ge \phi \ge K\frac {\delta} {\bar{E}} \quad \textup{on} \quad [\frac \tau 2, \tau) \times (\bar{B}_{\alpha r} \setminus \bar{B}_{r}),
\eeqn
which implies
\beqn
\nonumber
h \ge ge^{-\tau \bar{D}} \ge  K\frac {\delta} {\bar{E}}  e^{-\tau \bar{D}} \quad \hbox{on} \quad [\frac \tau 2, \tau) \times (\bar{B}_{\alpha r} \setminus \bar{B}_r).
\eeqn
Taking back to $f$ we have
\beqn
\nonumber
f \ge \frac 1 {\bar{E}} h \ge  K\frac {\delta} {\bar{E}^2}  e^{-\tau \bar{D}} \quad \hbox{on} \quad [\frac \tau 2, \tau) \times (\bar{B}_{\alpha r} \setminus \bar{B}_r),
\eeqn
we conclude the theorem.
\end{proof}

\bigskip
\section{Computation for $\phi_2(m)$ and $\varphi_p(m)$}
\label{secAB}
\setcounter{equation}{0}
\setcounter{theo}{0}
\begin{lem}\label{LB1}
Define
\beqn
\nonumber
\partial_t f :=\LL f  = \textup{div}_x(A(x, v) f ) + \textup{div}_v(B(x, v ) f ) + K \Delta_v f,
\eeqn
with
\beqn
\nonumber
A(x, v)  = -v +\Phi(x),
\eeqn
and $K>0$ a constant, then for any weight function $m$ we have
\bear\label{EB1}
\quad \int  (f (\LL g) + g (\LL f)) m^2&=&-2K \int \nabla_v f \cdot \nabla_v g m^2 + 2\int  f g \phi_2(m)m^2,
\eear
with
\bear
\nonumber
\phi_2(m)  &=& v \cdot \frac {\nabla_x m} m - \Phi(x) \cdot \frac {\nabla_x m} {m} + \frac 1 2 \textup{div}_x \Phi(x) +  K\frac {|\nabla_v m|^2} {m^2}
\\ \nonumber
 &+& K \frac{\Delta_v m} {m} - B(x, v) \cdot \frac{\nabla_v m } {m} + \frac 1 2 \textup{div}_v B(x, v).
\eear
Also we have for $ p \in [1, \infty]$
\bear\label{EB2}
\quad \quad  \int \sign f |f|^{p-1} \LL f m^p=- K\int |\nabla_v( m f )|^2 |f|^{p-2} m^{p-2} + \int  |f|^p \varphi_p(m)m^p,
\eear
with
\bear
\nonumber
\varphi_p(m)  &=& v \cdot \frac {\nabla_x m} m - \Phi(x) \cdot \frac {\nabla_x m} {m} + ( 1 -\frac 1 p) \textup{div}_x \Phi(x) +  2K(1-\frac 1 p) \frac {|\nabla_v m|^2} {m^2}
\\ \nonumber
&&+ K (\frac 2 p -1)\frac{\Delta_v m} {m} - B(x, v) \cdot \frac{\nabla_v m } {m} + (1-\frac 1 p) \textup{div}_v B(x, v),
\eear
where we use $\int f$ in place of $\int_{\R^d \times \R^d} f dx dv$ for short.
\end{lem}

\sk
\begin{proof}
\noindent  Define
\beqn
\nonumber
\TT f= -v \cdot \nabla_x f,
\eeqn
we have
\bear
\nonumber
\int f( \TT g) m^2 +\int (\TT f) g m^2 = \int \TT (fg) m^2 = -\int f g \TT (m^2) = -2\int f g m^2 \frac{\TT m } {m},
\eear
 for the  term with operator $\Delta$ we have
\bear
\nonumber
\int  (f \Delta_v g  +\Delta_v f g) m^2 &=& -  \int \nabla_v (fm^2)  \cdot \nabla_v g + \nabla_v (gm^2)  \cdot \nabla_v f
\\ \nonumber
&=&- 2\int \nabla_v f \cdot \nabla_v g m^2 +\int f g \Delta_v (m^2) 
\\ \nonumber
&=&- 2\int \nabla_v f \cdot \nabla_v g m^2 + 2 \int f g (|\nabla_v m|^2 + \Delta_v m m).
\eear
For the other terms, using integration by parts we deduce
\bear
\nonumber
&&\int f \hbox{div}_v (B(x, v)  g )m^2 + g \hbox{div}_v(B(x, v ) f) m^2 
\\ \nonumber
&=&\int f B(x, v) \cdot \nabla_v g m^2 + g B(x, v) \cdot \nabla_v f m^2  +2\hbox{div}_v B(x, v) f g m^2
\\ \nonumber
&=& - \int fg \nabla_v \cdot (B(x, v)  m^2) +2\hbox{div}_v B(x, v)  fg m^2
\\ \nonumber
&=&\int  - 2fg B(x, v) \cdot \frac {\nabla_v   m} {m} m^2 +\hbox{div}_v B(x, v)  fg m^2.
\eear
Similarly
\bear
\nonumber
&&\int f \hbox{div}_x(\Phi(x) g )m^2 + g \hbox{div}_x(\Phi(x) f) m^2 
\\ \nonumber
&=&\int  - 2fg \Phi(x) \cdot \frac {\nabla_v   m} {m} m^2 +\hbox{div}_x \Phi(x) fg m^2,
\eear
so (\ref{EB1}) are proved by combining the terms above.
For (\ref{EB2}) we compute
\bear
\nonumber
C_1&:=& \int \sign f |f|^{p-1} m^p \Delta_v f =- \int \nabla_v (\sign f |f|^{p-1}m^p) \cdot \nabla_v f
\\ \nonumber
&=&\int - (p-1)|\nabla_v f|^2 |f|^{p-2} m^{p} - \frac 1 p \int \nabla_v |f|^p  \cdot \nabla_v (m^p).
\eear
Using $\nabla_v(m f) =m \nabla_v f + f \nabla_v m$, we deduce
\bear
\nonumber
C_1&=&-(p-1) \int |\nabla_v(mf)|^2|f|^{p-2}m^{p-2}+(p-1)\int |\nabla_v m|^2|f|^p m^{p-2}
\\ \nonumber
&&+ \frac{2(p-1)} {p^2} \int \nabla_v (|f|^p) \cdot  \nabla_v (m^{p}) - \frac 1 p \int \nabla_v (|f|^p) \cdot \nabla_v (m^p)
\\ \nonumber
&=&-(p-1) \int |\nabla_v(m f)|^2|f|^{p-2} m^p+(p-1)\int |\nabla_v m|^2|f|^p m^{p-2}
\\ \nonumber
&& - \frac {p-2} {p^2} \int |f|^p \Delta_v m^p.
\eear
Using that $\Delta_v m^p = p \Delta_v m\  m^{p-1} + p(p-1)| \nabla_v m|^2m^{p-2}  $, we obtain
\bear
\nonumber
C_1&=&-(p-1) \int |\nabla_v(mf)|^2|f|^{p-2}m^{p-2} 
\\ \nonumber
&&+ \int |f|^p m^p \left[ \left(\frac 2 p -1\right) \frac { \Delta_v m} {m} + 2\left(1-\frac 1 p\right)\frac {|\nabla_v m|^2} {m^2} \right].
\eear
For the other terms we have
\bear
\nonumber
&&\int \sign f |f|^{p-1} \hbox{div}_v (B(x, v)  f )m^p 
\\ \nonumber
&=& - \int \frac 1 p |f|^p \hbox{div}_v  (B(x, v)  m^p) +\hbox{div}_v B(x, v) |f|^p  m^p
\\ \nonumber
&=&\int  - |f|^p  B(x, v) \cdot \frac {\nabla_v   m} {m} m^p +(1- \frac 1 p)\hbox{div}_v B(x, v)  fg m^p,
\eear
similarly
\bear
\nonumber
&&\int \sign f |f|^{p-1} \hbox{div}_x (A(x, v)  f )m^p 
\\ \nonumber
&=&\int  - |f|^p  A(x, v) \cdot \frac {\nabla_x   m} {m} m^p +(1- \frac 1 p)\hbox{div}_x A(x, v)  fg m^p
\\ \nonumber
&=& \int |f|^p (v- \Phi(x)) \cdot \frac {\nabla_x   m} {m} m^p + (1- \frac 1 p)\hbox{div}_x\Phi(x) fg m^p.
\eear
Gathering all the terms (\ref{EB2}) is proved.
\end{proof}

\medskip
{\bf Acknowledgment. }Â Â 
The author thanks to S. Mischler and J. A. Ca\~nizo for fruitful discussions on the full work of the paper. This work was supported by grants from R\'egion Ile-de-France the DIM program and BISMA, Tsinghua University.

\end{appendix}

\bigskip
\bigskip
 \signcc

\end{document}